\documentclass{article}%

\usepackage{amsmath,amssymb,amsfonts}%
\usepackage{theorem}%
\usepackage{color}%
\usepackage{graphicx}%
\usepackage{hyperref}%

\theoremstyle{change}%
\newtheorem{definition}{Definition:}[section]%
\newtheorem{proposition}[definition]{Proposition:}%
\newtheorem{theorem}[definition]{Theorem:}%
\newtheorem{lemma}[definition]{Lemma:}%
\newtheorem{corollary}[definition]{Corollary:}%
{\theorembodyfont{\rmfamily}\newtheorem{remark}[definition]{Remark:}}%

\newenvironment{proof}
  {{\bf Proof:}}
  {\qquad \hspace*{\fill} $\Box$}%

\newcommand{\fa}{\mathfrak{a}}%
\newcommand{\fg}{\mathfrak{g}}%
\newcommand{\fk}{\mathfrak{k}}%
\newcommand{\fm}{\mathfrak{m}}%
\newcommand{\fn}{\mathfrak{n}}%
\newcommand{\fp}{\mathfrak{p}}%
\newcommand{\fs}{\mathfrak{s}}%
\newcommand{\fz}{\mathfrak{z}}%
\newcommand{\Ad}{\operatorname{Ad}}%
\newcommand{\ad}{\operatorname{ad}}%
\newcommand{\tr}{\operatorname{tr}}%
\newcommand{\id}{\operatorname{id}}
\newcommand{\inner}{\operatorname{int}}%
\newcommand{\cl}{\operatorname{cl}}%
\newcommand{\fix}{\operatorname{fix}}%
\newcommand{\Sl}{\operatorname{Sl}}%
\newcommand{\tm}{\times}%
\newcommand{\ep}{\varepsilon}%
\newcommand{\st}{\operatorname{st}}%
\newcommand{\reg}{\operatorname{reg}}%
\newcommand{\diag}{\operatorname{diag}}%
\newcommand{\Gl}{\mathrm{Gl}}%
\newcommand{\SO}{\mathrm{SO}}%
\newcommand{\rma}{\mathrm{\bf a}}%
\newcommand{\rmd}{\mathrm{d}}%
\newcommand{\rme}{\mathrm{e}}%
\newcommand{\rmR}{\mathrm{\bf R}}%
\newcommand{\rmA}{\mathrm{\bf A}}%
\newcommand{\rmS}{\mathrm{\bf S}}%
\newcommand{\AC}{\mathcal{A}}%
\newcommand{\BC}{\mathcal{B}}%
\newcommand{\CC}{\mathcal{C}}%
\newcommand{\EC}{\mathcal{E}}%
\newcommand{\MC}{\mathcal{M}}%
\newcommand{\OC}{\mathcal{O}}%
\newcommand{\QC}{\mathcal{Q}}%
\newcommand{\RC}{\mathcal{R}}%
\newcommand{\SC}{\mathcal{S}}%
\newcommand{\UC}{\mathcal{U}}%
\newcommand{\WC}{\mathcal{W}}%
\newcommand{\T}{\mathbb{T}}%
\newcommand{\E}{\mathbb{E}}%
\newcommand{\F}{\mathbb{F}}%
\newcommand{\R}{\mathbb{R}}%
\newcommand{\Z}{\mathbb{Z}}%
\newcommand{\Ly}{\mathrm{Ly}}%
\newcommand{\Mo}{\mathrm{Mo}}%
\newcommand{\inv}{\mathrm{inv}}%
\newcommand{\dist}{\mathrm{dist}}%
\newcommand{\tp}{\mathrm{top}}%
\newcommand{\length}{\mathrm{length}}%
\renewcommand{\sl}{\mathfrak{sl}}%

\begin{document}

\title{Lyapunov exponents and partial hyperbolicity of chain control sets on flag manifolds}%
\author{A.~Da Silva\thanks{Imecc - Unicamp, Departamento de Matem\'atica, Rua S\'ergio Buarque de Holanda, 651, Cidade Universit\'aria Zeferino Vaz 13083-859, Campinas - SP, Brasil; e-mail: ajsilva@ime.unicamp.br; AS was supported by FAPESP Grant 2016/11135-2 and a Guest Scientist Grant from the University of Passau, where part of this work was done.} \and C.~Kawan\thanks{Universit\"at Passau, Fakult\"at f\"ur Informatik und Mathematik, Innstra{\ss}e 33, 94032 Passau, Germany; e-mail: christoph.kawan@uni-passau.de}}
\date{}%
\maketitle%

\begin{abstract}
For a right-invariant control system on a flag manifold $\F_{\Theta}$ of a real semisimple Lie group, we relate the $\fa$-Lyapunov exponents to the Lyapunov exponents of the system over regular points. Moreover, we adapt the concept of partial hyperbolicity from the theory of smooth dynamical systems to control-affine systems, and we completely characterize the partially hyperbolic chain control sets on $\F_{\Theta}$.%
\end{abstract}
	

	

\section{Introduction}%

In this paper, we establish the concept of partial hyperbolicity as a property of controlled invariant sets of control-affine systems and we study a class of systems for which we can characterize the partially hyperbolic chain control sets completely. The notion of partial hyperbolicity was first introduced by Brin and Pesin \cite{BPe} within the theory of smooth dynamical systems. Generalizing the notion of uniform hyperbolicity, partial hyperbolicity is characterized by a splitting of the tangent bundle into three invariant subbundles, two of which form a uniformly hyperbolic splitting and the third one (the center bundle) lying strictly in between the other two in terms of growth rates. That is, in the center directions any expansion or contraction is uniformly slower than the expansion and contraction in the unstable and stable directions, respectively. The focus of the theory of partially hyperbolic dynamical systems is on indecomposibility properties such as ergodicity and topological transitivity, in particular the persistence of these properties under perturbations. For an overview of this theory the reader is referred to the excellent survey \cite{HPe}.%

The concept of uniform hyperbolicity for control systems was studied in \cite{CDu,DS1,DS2,DS4,Ka2}. In \cite{CDu}, controllability and robustness results were proved for chain control sets with a uniformly hyperbolic structure. In \cite{DS1,Ka2}, the authors derived a formula for the invariance entropy of a uniformly hyperbolic control set. In \cite{DS4}, it was proved that the invariance entropy of such sets depends continuously on parameters. A large class of examples of uniformly hyperbolic chain control sets was provided in \cite{DS2}. The main result of \cite{DS2} yields a complete classification of the uniformly hyperbolic chain control sets of invariant systems on flag manifolds of semisimple Lie groups. In the paper at hand, we extend this analysis with the aim to characterize the partially hyperbolic chain control sets. In \cite{DS2}, it has already been shown that every chain control set of an invariant system allows for a decomposition of its extended tangent bundle into three continuous invariant subbundles, two of which form a uniformly hyperbolic splitting. Hence, the remaining work is to single out those cases in which the expansion and contraction rates in center directions are uniformly strictly smaller than those in the stable and unstable directions.%

The paper is structured as follows. Section \ref{sec_dynctrl} gives a brief introduction to dynamical and control systems and introduces the concept of partial hyperbolicity for controlled invariant subsets of the state space. In Section \ref{sec_invcs}, the main concepts and results concerning flows on principal bundles with semisimple structural group are presented. Some technical lemmas used in the proof of the main results are also stated and proved in this section. In Section \ref{sec_le}, we prove our main result about Lyapunov exponents. It is shown that over regular points the Lyapunov exponents of invariant systems on flag manifolds can be recovered from a vectorial exponent, the so-called `$\fa$-Lyapunov exponent'. Moreover, the decomposition of the tangent bundle into three subbundles, present over the chain control sets of invariant systems, allows us to define the equivalent to a Morse spectrum for each of these subbundles and we show that this spectral set contains all the asymptotic information of the system provided by the Lyapunov exponents in the subbundle directions. Section \ref{sec_ph} is devoted to the study of partial hyperbolicity. Here we show that a complete characterization of this property is possible if one knows the Morse spectrum in the directions of the subbundles. We show that a chain control set of an invariant system on a flag manifold is partially hyperbolic if and only if there is no intersection between the Morse spectra of the subbundles. In Section \ref{sec_ex}, we analyze the possibilities for partially hyperbolic chain control sets of invariant systems on the flag manifolds of $G = \Sl(3,\R)$, and we present an example on the 2-torus, where we can explicitly verify partial hyperbolicity. Section \ref{sec_ie} is devoted to the application of the above characterization to the estimation of the invariance entropy of invariant systems on flag manifolds from below. It is shown that if the Morse spectrum associated with the center bundle is trivial, then the infimum, on the associated Morse set, of the exponential growth rate of the unstable determinant is a lower bound for the invariance entropy, generalizing the previous result in \cite{DS2} proved for the case of a vanishing center bundle. Some concepts and technical lemmas that are used in the main results are stated in an appendix, Section \ref{sec_ap}.%

{\bf Notation:} We write $\R$ for the reals, $\Z$ for the integers and $\Z_+ = \{ n \in \Z : n \geq 0\}$. If $M$ is a smooth manifold, $T_xM$ denotes the tangent space to $M$ at $x$. We write $(\rmd f)_x:T_xM \rightarrow T_{f(x)}N$ for the derivative of a smooth map $f$ between manifolds $M$ and $N$. If $A$ is a subset of some metric space, we write $\cl A$ for its closure and $\inner A$ for its interior, respectively. We use the notation $|\cdot|$ for vector norms and $\|\cdot\|$ for the associated operator norms. We also write $m(\cdot)$ for the conorm of an operator, i.e., $m(A) = \min_{|x|=1}|Ax|$. The natural logarithm of a real number $x>0$ is denoted by $\log x$, and additionally we put $\log 0 := -\infty$. Moreover, we write $\log^+ x = \max(0,\log x)$.%

\section{Dynamical and control systems}\label{sec_dynctrl}%
	
In this section, we recall well-known facts about flows on metric spaces and control-affine systems that can be found, e.g., in Colonius and Kliemann \cite{CKl}.%

\subsection{Morse decompositions}%

Consider a continuous flow $\phi:\R\tm X\rightarrow X$, $(t,x)\mapsto\phi_t(x)$, on a compact metric space $(X,d)$. A compact set $K\subset X$ is called \emph{isolated invariant} if it is invariant, i.e., $\phi_t(K)\subset K$ for all $t\in\R$, and if there is a neighborhood $N$ of $K$ such that the implication%
\begin{equation*}
	\phi_t(x) \in N \mbox{\ for all\ } t\in\R \quad\Rightarrow\quad x\in K%
\end{equation*}
holds. A \emph{Morse decomposition} of $\phi$ is a finite collection $\{\MC_1,\ldots,\MC_n\}$ of nonempty pairwise disjoint isolated invariant compact sets satisfying:%
\begin{enumerate}
\item[(a)] For all $x\in X$, the $\alpha$- and $\omega$-limit sets $\alpha(x)$ and $\omega(x)$, respectively, are contained in $\bigcup_{i=1}^n\MC_i$.%
\item[(b)] Suppose there are $\MC_{j_0},\ldots,\MC_{j_l}$ and $x_1,\ldots,x_l\in X\backslash\bigcup_{i=1}^n\MC_i$ with%
\begin{equation*}
	\alpha(x_i) \in \MC_{j_{i-1}} \mbox{\quad and\quad } \omega(x_i) \in \MC_{j_i}%
\end{equation*}
for $i=1,\ldots,l$. Then $\MC_{j_0}\neq\MC_{j_l}$.%
\end{enumerate}
The elements of a Morse decomposition are called \emph{Morse sets}. We say that a compact invariant set $A$ is an \emph{attractor} if it admits a neighborhood $N$ such that $\omega(N)=A$. A \emph{repeller} is a compact invariant set $R$ which has a neighborhood $N^*$ with $\alpha(N^*) = R$. A Morse decomposition is \emph{finer} than another one if every element of the second one contains one of the first.%
	
Morse decompositions are related to the chain recurrent set of the flow. Recall that an \emph{$(\ep,T)$-chain} from $x\in X$ to $y\in X$ is given by an integer $n\geq1$, $n+1$ points $x=x_0,x_1,\ldots,x_n=y \in X$, and $n$ times $T_0,\ldots,T_{n-1}\geq T$ such that $d(\phi_{T_i}(x_i),x_{i+1})<\ep$ for $i=0,1,\ldots,n-1$. A subset $Y\subset X$ is \emph{chain transitive} if for all $x,y\in Y$ and $\ep,T>0$ there is an $(\ep,T)$-chain from $x$ to $y$. A point $x\in X$ is \emph{chain recurrent} if for all $\ep,T>0$ there is an $(\ep,T)$-chain from $x$ back to $x$. The \emph{chain recurrent set $\RC = \RC(\phi)$} is the set of all chain recurrent points. Then the connected components of $\RC$ coincide with the maximal chain transitive subsets, which are also called the \emph{chain recurrent components} of $\phi$. A finest Morse decomposition for $\phi$ exists iff there are only finitely many chain recurrent components. In this case, the Morse sets are the chain recurrent components.%
	
\subsection{Control-affine systems}\label{subsec_cas}%

A \emph{control-affine system} is a family of ordinary differential equations of the form%
\begin{equation}\label{eq_cas}
	\dot{x}(t) = f_0(x(t)) + \sum_{i=1}^mu_i(t)f_i(x(t)),\quad u\in\UC,%
\end{equation}
where $f_0,f_1,\ldots,f_m$ are $\CC^1$-vector fields on a smooth manifold $M$, the \emph{state space} of the system. The set $\UC$ of \emph{admissible control functions} is given by%
\begin{equation*}
	\UC = \left\{u:\R\rightarrow\R^m\ :\ u \mbox{ is measurable with } u(t)\in U \mbox{ a.e.}\right\},%
\end{equation*}
where $U\subset\R^m$ is a compact and convex set with $0\in\inner U$. For each $u\in\UC$ and $x\in M$ the corresponding (Carath\'eodory) differential equation \eqref{eq_cas} has a unique solution $\varphi(t,x,u)$ with initial value $x = \varphi(0,x,u)$. The systems considered in this paper all have globally defined solutions, which give rise to a map%
\begin{equation*}
	\varphi:\R \tm M \tm \UC \rightarrow M,\quad (t,x,u) \mapsto \varphi(t,x,u),%
\end{equation*}
called the \emph{transition map} of the system. We also write $\varphi_{t,u}:M\rightarrow M$, $x\mapsto\varphi(t,x,u)$. If the vector fields $f_0,f_1,\ldots,f_m$ are of class $\CC^k$, then $\varphi$ is of class $\CC^k$ with respect to the state variable and the corresponding partial derivatives of order $1$ up to $k$ depend continuously on $(t,x,u) \in \R\tm M\tm\UC$ (see \cite[Thm.~1.1]{Kaw}).%
	
The transition map $\varphi$ is a cocycle over the shift flow%
\begin{equation*}
	\theta:\R \tm \UC \rightarrow \UC,\quad (t,u) \mapsto \theta_tu = u(\cdot + t),%
\end{equation*}
i.e., it satisfies $\varphi(t+s,x,u) = \varphi(s,\varphi(t,x,u),\theta_tu)$ for all $t,s\in\R$, $x\in M$, $u\in\UC$. Together with the shift flow, $\varphi$ constitutes a continuous skew-product flow%
\begin{equation*}
	\phi:\R \tm \UC \tm M \rightarrow \UC \tm M,\quad (t,u,x) \mapsto (\theta_tu,\varphi(t,x,u)),%
\end{equation*}
where $\UC$ is endowed with the weak$^*$-topology of $L^{\infty}(\R,\R^m) = L^1(\R,\R^m)^*$, which gives $\UC$ the structure of a compact metrizable space. The flow $\phi$ is called the \emph{control flow} of the system (cf.~\cite{CKl,Kaw}). The base flow $\theta$ is chain transitive.%
	
In the following, we fix a metric $d$ on $M$. We call a set $E\subset M$ \emph{all-time controlled invariant} if for each $x\in E$ there exists $u\in\UC$ with $\varphi(\R,x,u)\subset E$. The \emph{all-time lift} $\EC$ of $E$ is defined by%
\begin{equation*}
	\EC := \left\{ (u,x) \in \UC \tm M\ :\ \varphi(\R,x,u) \subset E \right\},%
\end{equation*}
which is easily seen to be $\phi$-invariant. For points $x,y\in M$ and numbers $\ep,\tau>0$, a \emph{controlled $(\ep,\tau)$-chain from $x$ to $y$} is given by an integer $n\geq 1$, points $x_0,\ldots,x_n\in M$, controls $u_0,\ldots,u_{n-1}\in\UC$, and times $t_0,\ldots,t_{n-1}\geq\tau$ such that $x_0 = x$, $x_n = y$, and $d(\varphi(t_i,x_i,u_i),x_{i+1}) < \ep$ for $i=0,\ldots,n-1$. A set $E\subset M$ is called a \emph{chain control set} if it is maximal with the following properties:%
\begin{enumerate}
\item[(A)] $E$ is all-time controlled invariant%
\item[(B)] For all $x,y\in E$ and $\ep,\tau>0$ there exists a controlled $(\ep,\tau)$-chain from $x$ to $y$ in $M$.%
\end{enumerate}
Every chain control set is closed. The all-time lift $\EC$ of a chain control set $E$ is a maximal $\phi$-invariant chain transitive set of the control flow $\phi$. Conversely, if $\EC\subset\UC\tm M$ is a maximal $\phi$-invariant chain transitive set, then the projection $E = \left\{x\in M : \exists u\in\UC \mbox{ with } (u,x)\in\EC\right\}$ of $\EC$ to $M$ is a chain control set (cf.~\cite[Thm.~4.1.4]{CKl}).%
	
Now we give the definition of a partially hyperbolic all-time controlled invariant set, for which we need to equip $M$ with a Riemannian metric.%

\begin{definition}\label{def_phs}
Let $E \subset M$ be a compact all-time controlled invariant set with all-time lift $\EC$. We call $E$ {\it partially hyperbolic} if there exists a decomposition%
\begin{equation*}
  T_xM = E^-(u,x) \oplus E^0(u,x) \oplus E^+(u,x),\quad \forall (u,x) \in \EC%
\end{equation*}
into linear subspaces with $\dim E^-(u,x) + \dim E^+(u,x) \geq 1$ for all $(u,x) \in \EC$, satisfying the following conditions:%
\begin{enumerate}
\item[(i)] The subspaces $E^i(u,x)$, $i = -,0,+$, depend continuously on $(u,x)$.%
\item[(ii)] The subspaces $E^i(u,x)$, $i = -,0,+$, define invariant subbundles in the sense that%
\begin{equation*}
  (\rmd\varphi_{t,u})_x E^i(u,x) = E^i(\phi_t(u,x)),\quad \forall (u,x) \in \EC,\ t \in \R.%
\end{equation*}
\item[(iii)] There exist constants $c>0$, $\lambda<0<\mu$ and $\lambda<\lambda'\leq\mu'<\mu$ such that for all $t\geq 0$ and $(u,x)\in\EC$ we have%
\begin{align*}
	\left|(\rmd\varphi_{t,u})_x v\right| &\leq c\rme^{\lambda t}|v| \mbox{\quad for all\ } v \in E^-(u,x),\\
	c^{-1}\rme^{\lambda't}|v| \leq &\left|(\rmd\varphi_{t,u})_x v\right| \leq c\rme^{\mu' t}|v| \mbox{\quad for all\ } v \in E^0(u,x),\\
	c^{-1}\rme^{\mu t}|v| &\leq \left|(\rmd\varphi_{t,u})_x v\right| \mbox{\quad for all\ } v \in E^+(u,x).%
\end{align*}
\end{enumerate}
\end{definition}

It is not hard to see that this definition is independent of the Riemannian metric imposed on $M$ due to the compactness of $E$. The assumption that $\dim E^- + \dim E^+ \geq 1$ excludes trivial cases. If $\dim E^0(u,x) = 0$ for all $(u,x) \in \EC$, we also call $E$ \emph{uniformly hyperbolic (without center bundle)}. If $\EC$ is connected, which holds, e.g., if $E$ is a chain control set, it easily follows that the dimensions of the subspaces $E^i(u,x)$ are constant on $\EC$.%

\section{Invariant control systems}\label{sec_invcs}
	
\subsection{Semisimple theory}

Standard references for the theory of semisimple Lie groups and their flag manifolds are Duistermat-Kolk-Varadarajan \cite{DKV}, Helgason \cite{Hel}, Knapp \cite{Kna} and Warner \cite{War}. In the following, we only provide a brief review of the concepts used in this paper, see also \cite{DS2}.%

Let $G$ be a connected semisimple non-compact Lie group $G$ with finite center and Lie algebra $\fg$. We choose a Cartan involution $\zeta:\fg\rightarrow\fg$ and denote by $B_{\zeta}(X,Y) = -C(X,\zeta(Y))$ the associated inner product, where $C(X,Y) = \tr(\ad(X)\ad(Y))$ is the Cartan-Killing form. If $\fk$ and $\fs$ stand, respectively, for the eigenspaces of $\zeta$ associated with $-1$ and $1$, the Cartan decompositions of $\fg$ and $G$ are given, respectively, by%
\begin{equation*}
  \fg = \fk \oplus \fs\;\;\;\;\mbox{ and }\;\;\;\;G = KS, \;\;\;\;\mbox{ where }\;\;\;K = \exp\fk \mbox{ and } S = \exp\fs.%
\end{equation*}
Fix a maximal abelian subspace $\fa\subset\fs$ and denote by $\Pi$ the set of roots for this choice. If $\fn^+ := \sum_{\alpha\in\Pi^{+}}\fg_{\alpha}$, where $\Pi^+$ is the set of positive roots and%
\begin{equation*}
  \fg_{\alpha} = \left\{ X\in\fg\ :\ \ad(H)X = \alpha(H)X,\ \forall H \in \fa \right\}%
\end{equation*}
is the root space associated with $\alpha\in\Pi$, the Iwasawa decompositions of $\fg$ and $G$ are given, respectively, by%
\begin{equation*}
	\fg = \fk \oplus \fa \oplus \fn^+ \;\;\;\;\mbox{ and }\;\;\;\; G=KAN^+, \;\;\;\;\mbox{ where }\;\;N^+=\exp\fn^+ \mbox{ and }A=\exp\fa.%
\end{equation*}
The Weyl group $\WC$ is the group generated by the orthogonal reflections at the hyperplanes $\ker\alpha$, where $\alpha\in\Sigma$ and $\Sigma$ denotes the set of simple roots. Alternatively, $\WC=M^*/M$, where $M^*$ and $M$ are the normalizer and the centralizer of $\fa$ in $K$, respectively. The principal involution $w_0\in\WC$ is the only element in $\WC$ that satisfies $w_0\Pi^+=-\Pi^-$, where $(w\alpha)(H):=\alpha(w^{-1}H)$.%
	 
Let $\fa^+\subset\fa$ be the positive Weyl chamber associated with the above choices and consider $H\in\cl\fa^+$. The eigenspaces of $\ad(H)$ in $\fg$ are given by $\fg_{\alpha}$, $\alpha\in\Pi$. The centralizer of $H$ in $\fg$ is given by%
\begin{equation*}
  \fn^0_H := \sum_{\alpha\in\Pi:\ \alpha(H)=0}\fg_{\alpha}%
\end{equation*}
and the centralizer in $\fk$ by $\fk_H := \fk \cap \fn_H^0$.\footnote{Usually, the notation $\fz_H$ is used for the centralizer of $H$. However, the notation $\fn^0_H$ will be more convenient for our purposes.} They are, respectively, the Lie algebra of the centralizer of $H$ in $G$, $Z_H:=\{g\in G: \Ad(g)H=H\}$, and in $K$, $K_H=K\cap Z_H$. The positive and negative nilpotent subalgebras of type $H$ are given by%
\begin{equation*}
  \fn^+_H := \sum_{\alpha\in\Pi :\ \alpha(H)>0}\fg_{\alpha} \mbox{\quad and\quad} \fn^-_H := \sum_{\alpha\in\Pi :\ \alpha(H)<0}\fg_{\alpha}.%
\end{equation*}
The subgroup $N^{\pm}_H = \exp(\fn^{\pm}_H)$ is a normal subgroup of $N^{\pm}$ satisfying $N^{\pm} = N^{\pm}_HN^{\pm}(H)$, where $N^{\pm}(H) = \exp(\fn^{\pm}(H))$ is given by%
\begin{equation*}
  \fn^{\pm}(H) := \sum_{\alpha\in\Pi^{\pm}:\ \alpha(H)=0}\fg_{\alpha}.%
\end{equation*}
The parabolic subalgebra of type $H$ is given by%
\begin{equation*}
  \fp_H := \sum_{\alpha\in\Pi:\ \alpha(H)\geq 0}\fg_{\alpha}.%
\end{equation*}
Its associated subgroup $P_H$ is the normalizer of $\fp_H$ in $G$. The flag manifold of type $H$ is given by the orbit $\F_H := \Ad(G)\fp_H$ or, equivalently, by the homogeneous space $G/P_H$.%

The natural action of $G$ on $\F_H$ is given by $(g,x) \mapsto g \cdot x := \Ad(g)\fp_H$. For a fixed element $g\in G$, the derivative of the diffeomorphism $x \mapsto g \cdot x$ on $\F_H$ is denoted by%
\begin{equation*}
  (\rmd[g])_x:T_x\F_H \rightarrow T_{g\cdot x}\F_H.%
\end{equation*}
	 
Alternatively, we can associate the above subalgebras and subgroups to any subset $\Theta\subset\Sigma$ by considering $H\in\cl\fa^+$ such that $\Theta=\Theta(H):=\{\alpha\in\Sigma : \alpha(H)=0\}$. When this is the case, we will use $\Theta$ instead of $H$ as the subscript. Moreover, we will denote by $\langle\Theta\rangle$ the set of roots in $\Pi$ generated by linear combinations of the elements in $\Theta$.%

An element of $\fg$ of the form $Y = \Ad(g)H$ with $g\in G$ and $H\in\cl\fa^+$ is called a \emph{split element}. The flow $\exp(tH)$, induced by a split element $H\in\cl\fa^+$ on $\F_{\Theta}$, is given by $(t,\Ad(g)\fp_{\Theta}) \mapsto \Ad(\rme^{tH}g)\fp_{\Theta}$. The associated vector field can be shown to be a gradient vector field with respect to an appropriate Riemannian metric on $\F_{\Theta}$. The connected components of the fixed point set of this flow are given by%
\begin{equation*}
	 \fix_{\Theta}(H,w) = Z_H \cdot wb_{\Theta} = K_H \cdot wb_{\Theta},\quad w \in \WC.%
\end{equation*}
The sets $\fix_{\Theta}(H,w)$ are in bijection with the double coset space $\WC_{\Theta(H)}\backslash \WC/\WC_{\Theta}$, where $\WC_{\Theta(H)}$ and $\WC_{\Theta}$ are, respectively, the group generated by the reflections at $\ker\alpha$ for $\alpha\in \Theta(H)$ and $\alpha\in\Theta$. Each component $\fix_{\Theta}(H,w)$ is a compact connected submanifold of $\F_{\Theta}$. Define the negative parabolic subalgebra of type $H$ by%
\begin{equation*}
  \fp^-_H := \sum_{\alpha\in \Pi:\ \alpha(H) \leq 0}\fg_{\alpha}%
\end{equation*}
and the negative parabolic subgroup $P^-_H$ as the normalizer of $\fp^-_H$.%
	 
Each connected component of the fixed point set has a stable manifold given by%
\begin{equation*}
  \st_{\Theta}(H,w) = N^-_H\cdot \fix_{\Theta}(H, w)=P^-_H\cdot wb_{\Theta}%
\end{equation*}
whose union gives the Bruhat decomposition of $\F_{\Theta}$:%
\begin{equation*}
  \F_{\Theta} = \dot{\bigcup_{w\in\WC_{\Theta(H)}\setminus\WC/\WC_{\Theta}}}\st_{\Theta}(H,w).%
\end{equation*}
In the general case, when $Y = \Ad(g)H$ for $g\in G$ and $H\in\cl\fa^+$, we have%
\begin{equation*}
  \fix_{\Theta}(Y,w) = g \cdot \fix_{\Theta}(H,w) \;\;\;\;\mbox{ and }\;\;\;\;\st_{\Theta}(Y,w) = g \cdot \st_{\Theta}(H,w), \quad w \in \WC.%
\end{equation*}
Moreover, $P_Y:=gP_Hg^{-1}$, $P^-_Y:=gP^-_Hg^{-1}$, $N^{\pm}_Y:=gN^{\pm}_Hg^{-1}$ and $Z_Y:=gZ_Hg^{-1}$.%
	 	
The next lemma relates the fixed point components to the stable manifolds.%
	
\begin{lemma}\label{equalityfix}
If $H_1,H_2\in\cl(\fa^+)$ satisfy $\Theta_1\subset\Theta_2$, where $\Theta_i=\Theta(H_i), \;i=1, 2$, then%
\begin{equation*}
  \fix_{\Theta}(H_2,w)\subset\bigcup_{s\in\WC_{\Theta_2}}\st_{\Theta}(H_1,sw).%
\end{equation*}
\end{lemma}
	
\begin{proof}
It suffices to show that%
\begin{equation}\label{eq_claim}
  \fix_{\Theta}(H_2,w) \cap \st_{\Theta}(H_1,w')\neq\emptyset \quad\Rightarrow\quad w'\in \WC_{\Theta_2}w\WC_{\Theta}.%
\end{equation}
Indeed, if this holds, then%
\begin{align*}
&\fix_{\Theta}(H_2,w) = \fix_{\Theta}(H_2,w) \cap \dot{\bigcup_{w'\in\WC_{\Theta_1}\setminus\WC/\WC_{\Theta}}}\st_{\Theta}(H_1, w') = \dot{\bigcup_{w'\in\WC_{\Theta_1}\setminus\WC/\WC_{\Theta}}}[\fix_{\Theta}(H_2,w) \cap \st_{\Theta}(H_1, w')]\\
	&\subset \bigcup_{s\in\WC_{\Theta_2}} \dot{\bigcup_{w'\in\WC_{\Theta_1}\setminus\WC/\WC_{\Theta}}}[\fix_{\Theta}(H_2,w) \cap \st_{\Theta}(H_1,sw')] \subset \bigcup_{s\in\WC_{\Theta_2}} [\fix_{\Theta}(H_2,w) \cap \st_{\Theta}(H_1,sw)] \subset \bigcup_{s\in\WC_{\Theta_2}}\st_{\Theta}(H_1,sw).%
\end{align*}
To show \eqref{eq_claim}, let $x \in \fix_{\Theta}(H_2,w) \cap \st_{\Theta}(H_1,w')$ for some $w'\in\WC$. Since $\fix_{\Theta}(H_2,w)$ is invariant under the actions of elements of $A$ (using that $A \subset Z_{H_2}$), we have $\rme^{nH_1} \cdot x \in \fix_{\Theta}(H_2,w)$ for all $n$. On the other hand, since $x\in\st_{\Theta}(H_1,w')$, we can write $x=g\cdot z$ with $g\in N_{\Theta_1}^-$ and $z\in\fix_{\Theta}(H_1,w')$. Using that $z$ is an equilibrium of the flow $(\rme^{tH_1})_{t\in\R}$, this implies%
\begin{equation*}
  \lim_{n\rightarrow\infty}\rme^{nH_1} \cdot x = \lim_{n\rightarrow\infty}\rme^{nH_1}g \cdot z = \lim_{n\rightarrow\infty}\rme^{nH_1} g \rme^{-nH_1} \cdot z = z,%
\end{equation*}
since $g \in N_{\Theta_1}^-$ implies $\rme^{nH_1}g\rme^{-nH_1} \rightarrow 1$. (Write $g = \exp(X)$ with $X \in \fn_{\Theta}^-$. Then $\rme^{nH}g\rme^{-nH} = (C_{\rme^{nH}} \circ \exp)(X) = (\exp \circ \Ad(\rme^{nH}))(X)  = (\exp \circ \rme^{n\ad(H)})(X)$. By definition of $\fn_{\Theta}^-$, we can assume that $X \in \fg_{\alpha}$ for some $\alpha\in\Pi$ with $\alpha(H)<0$. This implies $\rme^{n\ad(H)}X = \sum_{k=0}^{\infty}\frac{1}{k!} (n\alpha(H))^k X = \rme^{n\alpha(H)}X \rightarrow 0$, implying $\rme^{nH}g\rme^{-nH} \rightarrow 1$.) Now $z \in \fix_{\Theta}(H_1,w') \subset \fix_{\Theta}(H_2,w')$, implying $z \in \fix_{\Theta}(H_2,w) \cap \fix_{\Theta}(H_2,w')$, and thus $\fix_{\Theta}(H_2,w) = \fix_{\Theta}(H_2,w')$. This concludes the proof, because it implies $w' \in \WC_{\Theta_2}w\WC_{\Theta}$ and hence proves the claim \eqref{eq_claim}.
\end{proof}
	
Finally, we briefly describe the construction of a $K$-invariant Riemannian metric on $\F_{\Theta}$. For any $x\in\F_{\Theta}$ let us consider the linear map%
\begin{equation*}
  \pi_x:\fg\rightarrow T_x\F_{\Theta}, \quad X\mapsto X(x) := \frac{\rmd}{\rmd t}_{|_{t=0}}\rme^{tX} \cdot x.%
\end{equation*}
The isotropy subalgebra at $x$ is $\fg_x:=\ker\pi_x$. For any $g\in G$ we have%
\begin{equation*}
	(\rmd[g])_x\circ\pi_x = \pi_{gx}\circ \Ad(g),%
\end{equation*}
and therefore $\Ad(g)\fg_x = \fg_{gx}$. If $\fg_x^{\perp}$ stands for the orthogonal complement of $\fg_x$ with respect to the $K$-invariant inner product $B_{\zeta}$, the $K$-invariance of $B_{\zeta}$ implies $\Ad(g)\fg_x^{\perp}=\fg_{gx}^{\perp}$ for any $g\in K$. Moreover, it is straightforward to see that $\pi_x$ restricted to $\fg_x^{\perp}$ is a linear isomorphism between $\fg_x^{\perp}$ and $T_x\F_{\Theta}$ and so we can consider in $T_x\F_{\Theta}$ the inner product%
\begin{equation*}
  \langle X(x),Y(x)\rangle_x := B_{\zeta}(X,Y), \;\; \forall X,Y\in\fg_x^{\perp}.%
\end{equation*}
	
We have the following result (see \cite[Prop.~3.1]{MPLS}).%
	
\begin{proposition}\label{K-invariant}
The inner product $\langle\cdot,\cdot\rangle_x$ defines a $K$-invariant Riemannian metric on $\F_{\Theta}$ such that the restriction of the map 
$\pi_x$ to $\fg_x^{\perp}$ is an isometry. Furthermore, for any $X\in\fg$ we have $|\pi_x(X)| \leq |X|$ with equality iff $X\in\fg_x^{\perp}$.
\end{proposition}
	
For any $x\in\F_{\Theta}$, let us denote by $P_x$ the orthogonal projection onto $\fg^{\perp}_x$ in $\fg$. By the above proposition it is easy to see that%
\begin{equation*}
	|X(x)| = |P_x(X)| = \inf_{Y\in\fg_x}|X - Y|.%
\end{equation*}

\begin{lemma}\label{auxiliar}
Let $Y\in\fg$ be a split-regular element and denote by $U$ an eigenspace of $\ad(Y)$. For any $x\in\mathrm{fix}_{\Theta}(Y,w)$, the spaces $\fg_x$ and $\fg_x^{\perp}$ are $\ad(Y)$-invariant. Moreover, if $U\cap \fg_x^{\perp}\neq\emptyset$ for some $x\in\mathrm{fix}_{\Theta}(Y,w)$, then $U\cap \fg_x^{\perp}\neq\emptyset$ for all $x\in\mathrm{fix}_{\Theta}(Y,w)$.%
\end{lemma}
	
\begin{proof}
In fact, since $\ad(Y)$ is self-adjoint, to prove the first statement it is enough to show that $\ad(Y)\fg_x\subset\fg_x$. For any $X\in\fg_x$, we have $\ad(Y)X\in\fg_x$ iff $\Ad\left(\rme^{tY}\right)X = \rme^{t\ad(Y)}X\in\fg_x$ for all $t\in\R$. To show the last equation, for a fixed $t\in\R$, let%
\begin{equation*}
 \gamma(s) := \exp\left(s\Ad(\rme^{tY})X \right)\cdot x,\quad \gamma:\R \rightarrow \F_{\Theta}.
\end{equation*}
Since $x \in \fix_{\Theta}(Y,w)$, this point is an equilibrium of the flow $(\rme^{tY})_{t\in\R}$, and hence%
\begin{align*}
  \dot{\gamma}(0) &= \frac{\rmd}{\rmd s}_{|s=0} \rme^{tY} \rme^{sX} \rme^{-tY} \cdot x = \frac{\rmd}{\rmd s}_{|s=0} \rme^{tY} \rme^{sX} \cdot x = (\rmd[\rme^{tY}])_x X(x) = 0,%
\end{align*}
because $X(x) = 0$. Hence, $\Ad(\rme^{tY})X \in \fg_x$, proving that $\fg_x$ is $\ad(Y)$-invariant.%

For the second statement, let $x,x' \in \fix_{\Theta}(Y,w)$. By the definition of $\fix_{\Theta}(Y,w)$, there exists $l \in K_Y$ with $x' = l \cdot x$. Using that $\ad(Y)$ commutes with $\Ad(g)$ for all $g\in Z_Y$, we find that%
\begin{equation}\label{eq_intersection_prop}
  U \cap \fg_{x'}^{\bot} = U \cap \Ad(l)\fg_x^{\bot} = \Ad(l)(U \cap \fg_x^{\bot}).%
\end{equation}
Here we use that $\Ad(l)U = U$, which is shown as follows. Take $X \in U$ and let $\ad(Y)X = \lambda X$. Then%
\begin{equation*}
  \ad(Y)\Ad(l)X = \Ad(l) \ad(\Ad(l^{-1})Y)X = \Ad(l) \ad(Y)X = \lambda \Ad(l)X,%
\end{equation*}
implying $\Ad(l)X \in U$. The second statement of the lemma follows immediately from \eqref{eq_intersection_prop}.%
\end{proof}
	
\subsection{Flows on flag bundles and $\fa$-Lyapunov exponents}
	
Let $\pi:Q\rightarrow X$ be a $G$-principal bundle over the compact metric space $X$ with a semisimple Lie group $G$. Then $G$ acts continuously from the right on $Q$, this action preserves the fibers, and is free and transitive on each fiber. In particular, this implies that each fiber is homeomorphic to $G$. An automorphism of $Q$ is a homeomorphism $\phi:Q\rightarrow Q$ which maps fibers to fibers and respects the right action of $G$ in the sense that $\phi(q \cdot g) = \phi(q) \cdot g$. For each set $\Theta\subset\Sigma$ of simple roots there is a flag bundle $\E_{\Theta} = Q \tm_G \F_{\Theta}$ with typical fiber $\F_{\Theta}$ given by $(Q\tm \F_{\Theta})/\!\sim$, where $(q_1,b_1)\sim (q_2,b_2)$ iff there exists $g\in G$ with $q_1 = q_2 \cdot g$ and $b_1 = g^{-1} \cdot b_2$. We write $\E$ for the maximal flag bundle $Q\tm_G \F$.%
	
Now let $\phi_n:Q\rightarrow Q$, $n\in\Z$, be a (discrete-time) flow of automorphisms whose base flow on $X$ is chain transitive. This flow induces a flow on each of the associated flag bundles $\E_{\Theta}$, which we also denote by $\phi_n$. The flow $\phi_n:\E_{\Theta}\rightarrow\E_{\Theta}$ has finitely many chain recurrent components and thus a finest Morse decomposition. The Morse sets can be described as follows.%

\begin{theorem}{(cf.~\cite[Thm.~9.11]{BSM}, \cite[Thm.~5.2]{SMLS})}\label{thm_morsesets}
\begin{enumerate}
\item[(i)] There exist $H_{\phi}\in\cl\fa^+$ and a continuous $\phi$-invariant map%
\begin{equation*}
	h_{\phi}:Q \rightarrow \Ad(G)H_{\phi},\quad h_{\phi}(\phi_n(q)) \equiv h_{\phi}(q),%
\end{equation*}
into the adjoint orbit of $H_{\phi}$, which is equivariant, i.e., $h_{\phi}(q\cdot g) = \Ad(g^{-1})h_{\phi}(q)$, $q\in Q$, $g\in G$. The induced flow on $\E_{\Theta}$ admits a finest Morse decomposition whose elements are given fiberwise by%
\begin{equation*}
	\MC_{\Theta}(w)_{\pi(q)} = q \cdot \fix_{\Theta}(h_{\phi}(q),w),\quad w\in\WC.%
\end{equation*}
The set $\Theta(\phi) := \Theta(H_{\phi}) = \{\alpha \in \Sigma : \alpha(H_{\phi}) = 0\}$ is called the \emph{flag type} of the flow $\phi$.
\item[(ii)] The induced flow on $\E_{\Theta}$ admits only one attractor component $\MC^+_{\Theta} = \MC_{\Theta}(1)$ and one repeller component $\MC^-_{\Theta} = \MC_{\Theta}(w_0)$. Moreover, the attractor component $\MC^+_{\Theta(\phi)}$ is given as the image of a continuous section $\sigma_{\phi}:X\rightarrow\E_{\Theta(\phi)}$, i.e.,%
\begin{equation*}
	\left(\MC^+_{\Theta(\phi)}\right)_x = \sigma_{\phi}(x)\quad \mbox{for all } x\in X.%
\end{equation*}
\end{enumerate}
\end{theorem}
	
For the Morse sets on the maximal flag bundle $\E$ we also write $\MC(w)$. Alternatively, the Morse sets can be described via a block reduction of $\phi$.%
	
\begin{proposition}{(\cite[Prop.~5.4]{SMLS})}\label{prop_morsesets}
The set $Q_{\phi} = h_{\phi}^{-1}(H_{\phi})$ is a $\phi$-invariant subbundle of $Q$ with structural group $Z_{\phi} := Z_{H_{\phi}}$, called a block reduction of $\phi$. There exists a $K_{\phi}$-reduction $R_{\phi} \subset Q_{\phi}$, i.e., a subbundle with structural group $K_{\phi} = Z_{\phi} \cap K$, and%
\begin{equation*}
	\MC_{\Theta}(w) = \left\{ q \cdot wb_{\Theta}\ :\ q\in Q_{\phi} \right\} = \left\{ r\cdot wb_{\Theta}\ :\ r \in R_{\phi} \right\}.%
\end{equation*}
\end{proposition}
	
In the following, we recall some results from \cite{ASM,ASM1}.%

Fix a Cartan and an Iwasawa decomposition $G = KS$ and $G = KAN^+$, respectively. There exists a $K$-reduction $R\subset Q$, i.e., a subbundle with structural group $K$. Cartan and Iwasawa decompositions of $Q$ are given, respectively, by $Q=R\cdot S$ and $Q = R \cdot AN^+$, and we can write each $q\in Q$ in a unique way as $q=r'\cdot s$ and $q = r \cdot hn$ with $r,r'\in R$, $h\in A$ and $n\in N^+$. We denote by $\rmR:Q \rightarrow R$, $\rmS:Q \rightarrow S$ and $\rmA:Q \rightarrow A$ the corresponding (continuous) projections from $Q$ to $R$, $S$ and $A$, respectively. The exponential map of $G$ maps $\fa$ bijectively onto $A$. Writing $\log$ for the inverse of $\exp_{|\fa}$, we define $\rma(q) := \log \rmA(q)$, $\rma:Q\rightarrow \fa$. Then also%
\begin{equation*}
	\phi_n^R:R \rightarrow R,\quad \phi_n^R(r) = \rmR(\phi_n(r)),%
\end{equation*}
is a flow, and a continuous additive cocycle over $\phi_n^R$ is given by%
\begin{equation*}
	\rma^{\phi}:\Z \tm R \rightarrow \fa,\quad \rma^{\phi}(n,r) := \rma(\phi_n(r)).%
\end{equation*}
In the following, by abuse of notation, we only write $\rma$ for $\rma^{\phi}$. Then $\rma$ induces a cocycle over the flow on the maximal flag bundle $\E = Q \tm_G \F = R \tm_K \F$ by $\rma(n,\xi) := \rma(n,r)$, $\xi = r \cdot b_0$.  The $\fa$-\emph{Lyapunov exponent} of $\phi_n$ in the direction of $\xi\in\E$ is defined by%
\begin{equation*}
  \lambda(\xi) := \lim_{n\rightarrow+\infty}\frac{1}{n}\,\rma(n,\xi)\in\fa, \quad \xi\in\E,%
\end{equation*}
when the limit exists. By the polar decomposition $G = K\cl(A^+)K$, we obtain the map $\rmA^+:Q\rightarrow \cl A^+$, defined by $\rmS(q) = k\rmA^+(q)k^{-1}$ with $k\in K$. We define the \emph{polar exponent} by%
\begin{equation*}
  \lambda^+(r) := \lim_{n\rightarrow+\infty}\frac{1}{n}\,\log \rmA^+(\phi_n(\xi))\in\cl\fa^+, \quad r\in R,%
\end{equation*}
when the limit exists. It turns out that $\lambda^+(r)$ is constant along the fibers and so we only write $\lambda^+(x)$, $x\in X$, and denote by $X_{\reg}$ the set of points for which this limit exists, called the \emph{set of regular points}. The next result from \cite{ASM} (see also \cite{ASM1}) assures the existence of the above limits.%
	
\begin{theorem}\label{Teo}
Let $\nu$ be an invariant Borel probability measure on $X$. Then the polar exponent $\lambda^+(x)$ exists for all $x$ in a set $\Omega\subset X$ of full measure, invariant under the base flow on $X$. Put $\E_{\Omega} := \pi_{\E}^{-1}(\Omega)$, where $\pi_{\E}:\E\rightarrow X$ is the projection. Then

\begin{enumerate}
\item[(i)] $\lambda(\xi)$ exists for every $\xi\in \E_{\Omega}$ and the map $\lambda:\E_{\Omega}\rightarrow\fa$ assume values in the finite set $\{w\lambda^+(x) : w\in\WC\}$, $x=\pi_{\E}(\xi)$;
\item[(ii)] The map $D:R_{\Omega}:=\pi^{-1}(\Omega)\cap R\rightarrow\fs$ given by%
\begin{equation*}
  D(r) := \lim_{n\rightarrow+\infty}\frac{1}{n}\log \rmS(r),%
\end{equation*}
satisfies:%
\begin{itemize}
\item $D$ is equivariant, that is $D(r\cdot k)=\Ad(k^{-1})D(r)$ for any $k\in K$;
\item $D(r) = \Ad(u_r)\lambda^+(\pi(r))$ for some $u_r\in K$;%
\item $\lambda(\xi) = w^{-1}\lambda^+(\pi(\xi))$ for any $\xi\in r\cdot\st(D(r),w)$.
\end{itemize}
\end{enumerate}

If $\nu$ is ergodic, then $\lambda^+(x) =: \lambda^+(\nu)$ is constant on $\Omega$.%
\end{theorem}

For any $w\in\WC$ and $x\in X$, we write $\st(x,w) := r\cdot \st(D(r),w)$, where $r\in R_{\Omega}$ is an arbitrary element of the fiber over $x\in\Omega$, and put%
\begin{equation*}
  \st^{\nu}(w) := \bigcup_{x\in \Omega} \st(x,w).%
\end{equation*}
Each such set is invariant, measurable and given by%
\begin{equation*}
  \st^{\nu}(w) = \{\xi\in E_{\Omega} : \lambda(\xi) = w^{-1}\lambda^+(\pi_{\E}(\xi))\}.%
\end{equation*}

Each Morse component $\MC(w)$ of the induced flow on $\E$ has a \emph{Morse spectrum} $\Lambda_{\Mo}(\MC(w),\rma)$ which is a compact convex subset of $\fa$ that coincides with the convex hull of%
\begin{equation*}
  \Lambda_{\Ly}(\MC(w),\rma) = \{\lambda(\xi) : \;\xi\in\MC(w),\ \lambda(\xi) \mbox{ exists}\}.%
\end{equation*}
The spectrum of the attractor Morse component $\Lambda_{\Mo}(\MC^+,\rma)$ is the only Morse spectrum meeting $\cl\fa^+$ and we have the following result.%

\begin{proposition}\label{prop_spectralprops}
\begin{itemize}
\item[(i)] $\Lambda_{\Mo}(\MC^+, \rma)$ is invariant under $\WC_{\Theta(\phi)}$ and%
\begin{equation*}
  \Lambda_{\Ly}(\MC^+, \rma)=\bigcup_{x\in X_{\mathrm{reg}}}\WC_{\Theta(\phi)}\lambda^+(x);%
\end{equation*}
\item[(ii)] $\alpha(\lambda)>0$ if $\alpha\in\Pi^+\setminus\langle\Theta(\phi)\rangle$ and $\lambda\in \Lambda_{\Mo}(\MC^+,\rma)$.%
\end{itemize}
\end{proposition}

Moreover, from \cite[Prop.~6.4]{ASM1} we know that%
\begin{equation}\label{eq_fixedpoints_incl}
  \fix_{\Theta}(D(r),w) \subset \fix_{\Theta}(h_{\phi}(r),w) \mbox{\quad for any } r \in R_{\Omega} \mbox{ and } w\in\WC.%
\end{equation}
From now on we will use the notation%
\begin{equation*}
  \st^{\nu}_{\Theta}(w) := \pi_{\Theta}(\st^{\nu}(w)),\quad \Lambda_{\Ly}^+ := \Lambda_{\Ly}(\MC^+,\rma),\quad \Lambda_{\Mo}^+ := \Lambda_{\Mo}(\MC^+,\rma).%
\end{equation*}
	
\subsection{Invariant systems on flag manifolds}
	
An invariant control system on $G$ is a control-affine system%
\begin{equation}\label{invG}
	\dot{x}(t) = X_0(x(t)) + \sum_{i=1}^mu_i(t)X_i(x(t)),\quad u\in\UC,%
\end{equation}
where the $X_i$ are right-invariant vector fields. We write $(t,g,u) \mapsto \varphi(t,g,u)$ for the transition map of this system and note that by right-invariance we have%
\begin{equation*}
  \varphi(t,gh,u) = \varphi(t,g,u)h \mbox{\quad for all\ } g,h \in G,\ t\in\R,\ u \in \UC.%
\end{equation*}
Since $\UC \times G\rightarrow\UC$ is a (trivial) principal bundle, we can apply the theory of the preceding subsection to the control flow of \eqref{invG} to characterize the Morse components of the control flow of the induced system%
\begin{equation}\label{inv}
  \Sigma_{\Theta}:\quad \dot{x}(t) = f_0(x(t)) + \sum_{i=1}^mu_i(t)f_i(x(t)),\quad u\in\UC,%
\end{equation}
on the flag manifold $\F_{\Theta}$, where for $i=0,\ldots,m$ we have $f_i(\pi_{\Theta}(g)) = (\rmd\pi_{\Theta})_gX_i(g)$ with the canonical projection $\pi_{\Theta}:G\rightarrow \F_{\Theta}$. For simplicity, we also write $\varphi$ for the transition map of the system on $\F_{\Theta}$. It will become clear from the context which transition map is considered.%

By the theory outlined in the preceding subsection, the chain control sets on $\F_{\Theta}$ are given by%
\begin{equation*}
  E_{\Theta}(w) = \pi^{\Theta}_2(\MC_{\Theta}(w)) = \bigcup_{u\in\UC}\fix_{\Theta}(\mathsf{h}(u),w),\quad w\in\WC,%
\end{equation*}
where $\pi^{\Theta}_2$ is the projection onto the second component of $\UC\times\F_{\Theta}$ and $\mathsf{h}(u) := h_{\phi}(u,1)$ (see \cite{DS2}). Moreover, for any $(u,x)\in\MC_{\Theta}(w)$ we have a decomposition%
\begin{equation*}
	T_x\F_{\Theta} = \EC^-_{\Theta,w}(u,x) \oplus \EC^0_{\Theta,w}(u,x) \oplus \EC^+_{\Theta,w}(u,x),%
\end{equation*}
where $\EC^i_{\Theta,w}(u,x) := \fn^i_{\mathsf{h}(u)}\cdot x$ and
\begin{itemize}
\item[1.] $(\rmd\varphi_{t,u})_x\EC^i_{\Theta,w}(u,x) = \EC^i_{\Theta,w}(\phi_t(u,x))$ for all $t\in\R$;%
\item[2.] There exist $C,\zeta>0$ such that for all $t\geq 0$,%
\begin{align}\label{eq_su_estimates}
\begin{split}
  |(\rmd\varphi_{t,u})_xV| &\leq C\rme^{-\zeta t}|V| \;\;\;\mbox{ for any }V\in \EC^-_{\Theta, w}(u,x),\\
	|(\rmd\varphi_{t,u})_xV| &\geq C^{-1}\rme^{\zeta t}|V| \;\;\;\mbox{ for any }V\in \EC^+_{\Theta,w}(u,x).%
\end{split}
\end{align}
\end{itemize}
	
Our aim is to relate the $\fa$-Lyapunov exponents to the usual Lyapunov exponents of the system \eqref{inv} and use this relation to study the asymptotic behavior on the center bundle $\EC_{\Theta,w}^0$. Let us now fix an invariant probability measure $\nu$ on the Borelians of $\UC$ and let $\Omega$ be the $\theta$-invariant set in $\UC$ of full measure given by Theorem \ref{Teo}.%
	
For any $u\in\Omega$ let us consider $D(u) := D(u,1)$. By the equivariance of $D$ we have $D(u,k) = \Ad(k^{-1})D(u)$, and therefore%
\begin{equation*}
  \st_{\Theta}^{\nu}(w) = \bigcup_{u\in \Omega} (u,k) \cdot \st_{\Theta}(D(u,k),w) = \bigcup_{u\in \Omega} (u,1) \cdot \st_{\Theta}(D(u),w) = \bigcup_{u\in \Omega} \{u\} \times \st_{\Theta}(D(u),w).%
\end{equation*}
 
Let us consider%
\begin{equation*}
  \MC^{\nu}_{\Theta}(w) := \bigcup_{u\in\Omega}\{u\} \tm \fix_{\Theta}(\mathsf{h}(u),w).%
\end{equation*}
	
\begin{proposition}\label{MorseStable}
For any $u\in\Omega$ there exists $k\in K$ such that $\mathsf{h}(u)=\Ad(k)H_{\phi}$ and $D(u)=\Ad(k)\lambda^+(u)$. Moreover, for any $\theta_1$-invariant measure $\nu$ on $\UC$ it holds that%
\begin{equation*}
  \MC_{\Theta}^{\nu}(w) \subset \bigcup_{s\in\WC_{\Theta(\phi)}}\st_{\Theta}^{\nu}(sw).%
\end{equation*}
\end{proposition}
	
\begin{proof} 
Let $v_1,v_2\in K$ such that $\Ad(v_1)H_{\phi} = \mathsf{h}(u)$ (see Theorem \ref{thm_morsesets}) and $\Ad(v_2)\lambda^+(u) = D(u)$ (see Theorem \ref{Teo}). Since $\fix(D(u),1) \subset \fix(\mathsf{h}(u),1)$, by \eqref{eq_fixedpoints_incl} we have
$v_1 K_{\lambda^+(u)}\cdot b_0 \subset v_2K_{H_{\phi}} \cdot b_0$ and thus $v_1\cdot b_0 = v_2\cdot lb_0$ for some $l\in K_{H_{\phi}}$. Therefore, $v_1^{-1}v_2l\in P \cap K = M$ implying that $v_1=v_2lm$ for some $m\in M$. Since $lm\in K_{\Theta(\phi)}$, by considering $v=v_1(lm)^{-1}$ we obtain%
\begin{equation*}
  \Ad(v)\lambda^+(u)=\Ad(v_2)\lambda^+(u)=D(u) \;\;\;\mbox{ and }\;\;\;\Ad(v)H_{\phi}=\Ad(v_1(lm)^{-1})H_{\phi}=\Ad(v_1)H_{\phi}=\mathsf{h}(u).%
\end{equation*}
Therefore,%
\begin{equation*}
  \fix_{\Theta}(\mathsf{h}(u), w)=v\cdot \fix_{\Theta}(H_{\phi}, w)\;\;\mbox{ and }\;\;\st_{\Theta}(D(u), w)=v\cdot \st_{\Theta}(\lambda^+(u),w)%
\end{equation*}
and since $\Theta(\lambda^+(u))\subset\Theta(\phi)$ (following from Proposition \ref{prop_spectralprops}(ii)), Lemma \ref{equalityfix} implies%
\begin{equation*}
  \fix_{\Theta}(H_{\phi}, w)\subset \bigcup_{s\in\WC_{\Theta(\phi)}}\st_{\Theta}(\lambda^+(u), sw)\;\;\;\;\;\mbox{ and so }\;\;\;\;\; \fix_{\Theta}(\mathsf{h}(u), w)\subset \bigcup_{s\in\WC_{\Theta(\phi)}}\st_{\Theta}(D(u),sw).%
\end{equation*}
Hence,%
\begin{equation*}
  \MC_{\Theta}^{\nu}(w) = \bigcup_{u\in \Omega}\{u\} \tm \fix_{\Theta}(\mathsf{h}(u), w)\subset \bigcup_{u\in\Omega}\bigcup_{s\in\WC_{\Theta(\phi)}}\{u\} \tm \st_{\Theta}(D(u),sw) = \bigcup_{s\in\WC_{\Theta(\phi)}}\st^{\nu}_{\Theta}(sw),%
\end{equation*}
as stated.
\end{proof}
	
We also have the following result.%
	
\begin{lemma}\label{technical}
Let $u\in\Omega$ and $x\in\fix_{\Theta}(\mathsf{h}(u),w)\cap\st_{\Theta}(D(u),sw)$ with $s\in\WC_{\Theta(\phi)}$. Then the following statements hold:%
\begin{enumerate}
\item[(i)] There exist $g\in Z_{\mathsf{h}(u)}\cap N^-_{D(u)}$ and $z\in\fix_{\Theta}(D(u),sw)$ such that $x = g\cdot z$;%
\item[(ii)] For any $V\in\EC^i_{\Theta,w}(u,x)$ there is $X\in \fn^i_{\mathsf{h}(u)}\cap\fg_z^{\perp}$ such that $V = (\rmd[g])_zX(z)$, where $i=+,0,-$.%
\end{enumerate}
\end{lemma}
		
\begin{proof}
(i) Consider $h\in N^-_{D(u)}$ and $z\in\fix_{\Theta}(D(u),sw)$ such that $x=h\cdot z$. Recall that $\Theta(\lambda^+(u))\subset\Theta(\phi)$, $D(u)=\Ad(k)\lambda^+(u)$ and $\mathsf{h}(u)=\Ad(k)H_{\phi}$ for some $k\in K$. We claim that we can write%
\begin{equation*}
  h = ag \mbox{\quad with\quad} a\in N_{\mathsf{h}(u)}^- \mbox{\ and\ } g\in N(\mathsf{h}(u))\cap N^-_{D(u)}\subset Z_{\mathsf{h}(u)}\cap N^-_{D(u)}.%
\end{equation*}
Since $N_{\mathsf{h}(u)} = kN_{H_{\phi}}k^{-1}$, $N_{D(u)} = kN_{\lambda^+(u)}k^{-1}$ and $Z_{\mathsf{h}(u)} = kZ_{H_{\phi}}k^{-1}$, it suffices to prove that if $x \in \fix_{\Theta}(H_{\phi},w)\cap\st_{\Theta}(\lambda^+(u),sw)$ and we write $x = h \cdot z$ with $h \in N^-_{\lambda^+(u)}$ and $z \in \fix_{\Theta}(\lambda^+(u),sw)$, then%
\begin{equation*}
  h = ag \mbox{\quad with\quad} a\in N_{H_{\phi}}^- \mbox{\ and\ } g\in N(H_{\phi})\cap N^-_{\lambda^+(u)}\subset Z_{H_{\phi}}\cap N^-_{\lambda^+(u)}.%
\end{equation*}
This is proved as follows: For any $H \in \cl\fa^+$ we have $N^- = N^-_H N^-(H)$. Hence, $h = ag$ with $a \in N^-_{H_{\phi}}$ and $g \in N^-(H_{\phi})$. Moreover, $g = a^{-1}h \in N^-_{H_{\phi}}N^-_{\lambda^+(u)} \subset N^-_{\lambda^+(u)}$, since $\Theta(\lambda^+(u)) \subset \Theta(\phi)$.%
		
The fact that $g\in Z_{\mathsf{h}(u)}$ implies $\rme^{t\mathsf{h}(u)}g = g\rme^{t\mathsf{h}(u)}$ for all $t\in\R$. Also, since $\fix_{\Theta}(D(u),sw)\subset\fix_{\Theta}(\mathsf{h}(u),sw) = \fix_{\Theta}(\mathsf{h}(u),w)$, we have $\rme^{t\mathsf{h}(u)}\cdot z = z$ and $\rme^{t\mathsf{h}(u)}\cdot x = x$ for all $t\in\R$. Therefore,%
\begin{equation*}
  x = \rme^{t\mathsf{h}(u)}\cdot x = \rme^{t\mathsf{h}(u)}h \cdot z =	\bigl(\rme^{t\mathsf{h}(u)}a\rme^{-t\mathsf{h}(u)}\bigr)g\cdot z\rightarrow g\cdot z \mbox{\quad as\ } t \rightarrow \infty,%
\end{equation*}
since $\rme^{\mathsf{h}(u)}a\rme^{-t\mathsf{h}(u)} \rightarrow 1$. This proves the first item.%

(ii) Since $\EC^i_{\Theta,w}(u,x) = \fn^i_{\mathsf{h}(u)}\cdot x$, we can write $V = Y(x)$ for $Y\in\fn^i_{\mathsf{h}(u)}$ such that $|V|=|Y|$. By item (i) we have $V = (\rmd[g])_z(\Ad(g^{-1})Y)(z)$, where $g\in Z_{\mathsf{h}(u)}\cap N^-_{D(u)}$ and $z\in\fix_{\Theta}(D(u),sw)$. Moreover, since $\fg=\fg_z^{\perp}\oplus\fg_z$, we can write $\Ad(g^{-1})Y=X+Z$ with $X\in\fg_z^{\perp}$ and $Z\in\fg_z$, and therefore $(\Ad(g^{-1})Y)(z)=X(z)$ implying $V = (\rmd[g])_zX(z)$. It only remains to show that $X\in\fn^i_{\mathsf{h}(u)}$. Since $g\in Z_{\mathsf{h}(u)}\cap N^-_{D(u)}$, we have%
\begin{equation*}
  \Ad(g^{-1})Y\in\Ad(g^{-1})\fn^i_{\mathsf{h}(u)}=\fn^i_{\mathsf{h}(u)}\implies X+Z\in\fn^i_{\mathsf{h}(u)}.%
\end{equation*}
On the other hand, there exists $k\in K$ such that $z=k\cdot wb_{\Theta}$ and $\mathsf{h}(u)=\Ad(k)H_{\phi}$, implying%
\begin{align*}
  \fn^i_{\mathsf{h}(u)}\cap (\fg_z^{\perp}\oplus\fg_z) &= \Ad(k)\left(\fn^i_{H_{\phi}}\cap (w\fn_{\Theta}^-\oplus w\fp_{\Theta})\right)\\
	&= \Ad(k)\left((\fn^i_{H_{\phi}}\cap w\fn_{\Theta}^-) \oplus(\fn^i_{H_{\phi}}\cap w\fp_{\Theta})\right)=(\fn^i_{\mathsf{h}(u)}\cap \fg_z^{\perp}) \oplus(\fn^i_{\mathsf{h}(u)}\cap\fg_z)%
\end{align*}
and consequently that $X,Z\in\fn^i_{\mathsf{h}(u)}$, as desired. To show the identity%
\begin{equation}\label{eq_decomposition}
  \fn^i_{H_{\phi}} = \left(\fn^i_{H_{\phi}}\cap (w\fn_{\Theta}^-\oplus w\fp_{\Theta})\right) = \left((\fn^i_{H_{\phi}}\cap w\fn_{\Theta}^-) \oplus(\fn^i_{H_{\phi}}\cap w\fp_{\Theta})\right),%
\end{equation}
observe that if $H_w := wH_{\Theta}$, then $\ad(H_w)$ and $\ad(H_{\Theta})$ commute, since $\fa$ is abelian. Hence, $\ad(H_w)\fn^i_{H_{\phi}} \subset \fn^i_{H_{\phi}}$, since $\fn^i_{H_{\phi}}$ is a sum of eigenspaces. Now we decompose $\fn^i_{H_{\phi}}$ into the sum of eigenspaces of $\ad(H_w)_{|\fn^i_{H_{\phi}}}$ associated with negative eigenvalues and the corresponding sum associated with nonnegative eigenvalues, $\fn^i_{H_{\phi}} = V_i^- \oplus V_i^{0+}$. Since $w\fn^-_{\Theta}$ and $w\fp_{\Theta}$ are the sums of eigenspaces associated with negative and nonnegative eigenvalues in $\fg$, we have $V_i^- \subset \fn^i_{H_{\phi}} \cap w\fn^-_{\Theta}$ and $V_i^{0+} \subset \fn^i_{H_{\phi}} \cap w\fp_{\Theta}$, implying%
\begin{equation*}
 \left((\fn^i_{H_{\phi}}\cap w\fn_{\Theta}^-) \oplus(\fn^i_{H_{\phi}}\cap w\fp_{\Theta})\right) \supset V_i^- \oplus V_i^{0+} = \fn^i_{H_{\phi}}.%
\end{equation*}
The other inclusion is trivial, hence \eqref{eq_decomposition} is proved.%
\end{proof}
	
\begin{remark}
Let us notice that $X\in\fg_z^{\perp}$ implies $|X(z)|=|X|$ by Proposition \ref{K-invariant}, and hence%
\begin{equation}\label{ring}
  \|(\rmd[g])_z\|^{-1}|V|\leq |X|\leq \|(\rmd[g])_z^{-1}\||V|.%
\end{equation}
\end{remark}
	
\section{Lyapunov exponents}\label{sec_le}
	
In this section, we show that for the points in $\st_{\Theta}^{\nu}(w)$ one can recover the Lyapunov exponents of the control system $\Sigma_{\Theta}$ from the $\fa$-Lyapunov exponents, where $w\in\WC$ and $\nu$ is any $\theta_1$-invariant probability measure on the Borelians of $\UC$. For the understanding of this section, we advise the reader to take a look at Subsection \ref{subsec_met} of the appendix first.%
	
\subsection{Regular Lyapunov exponents}

Let $\nu$ be a $\theta_1$-invariant probability measure on the Borelians $\BC$ of $\UC$ and consider the $\theta_1$-invariant set $\Omega$ of full measure so that Theorem \ref{Teo} holds. Let us consider the metric dynamical system $(\UC,\BC,\nu,(\theta_n)_{n\in\Z})$. If $A:\UC \rightarrow \Gl(\fg)$ is the random map given by $A(u) := \Ad(\varphi_{1,u}(e))$, then the linear cocycle $\psi(n,u)$ generated by $A$ satisfies%
\begin{equation*}
  \psi(n,u) = \Ad(\varphi_{n,u}(e))\quad \forall n \in \Z.%
\end{equation*}
Moreover, since $u \mapsto \Ad(\varphi_{1,u}(e))$ is continuous and $\UC$ is compact, we have $\log^+\|A\|,\log^+\|A^{-1}\| \in L^1(\UC,\nu,\BC)$, and hence there is a subset $\Omega' \subset \Omega$ of full measure such that Theorem \ref{MET} holds. In particular, by Lemma \ref{asymptotic} it holds that $\Psi(u) = \Ad(\rme^{D(u)})$, where%
\begin{equation*}
  \Psi(u) = \lim_{n\rightarrow+\infty}(\psi(n,u)^*\psi(n,u))^{1/2n}.%
\end{equation*}
Also, since $D(u)=\Ad(k)\lambda^+(u)$ for some $k\in K$ and $\lambda^+(u)\in\cl\fa^+$, the eigenvalues of $\ad(D(u))$ coincide with the eigenvalues of $\ad(\lambda^+(u))$ which are given by $\alpha(\lambda^+(u))$, $\alpha\in\Pi$. Let us denote by $\lambda_1(u)>\ldots>\lambda_{p(u)}(u)$ the distinct ones. The eigenspace of $\ad(\lambda^+(u))$ associated to $\lambda_i(u)$ is then given by%
\begin{equation}\label{decomposition}
	U_i^+(u) = \bigoplus_{\alpha\in\Pi: \;\alpha(\lambda^+(u))=\lambda_i(u)}\fg_{\alpha}%
\end{equation}
and consequently $U_i(u) = \Ad(k)U_i(u)^+$ is the corresponding eigenspace of $\ad(D(u))$. Also, the filtration $V_{p(u)}(u)\subset\cdots\subset V_1(u) = \fg$ given by Theorem \ref{MET} satisfies%
\begin{equation*}
  V_i(u) = \Ad(k)V^+_i(u), \;\;\mbox{ where  }\;\; V^+_i(u) = U^+_{p(u)}(u) \oplus \cdots \oplus U^+_i(u).%
\end{equation*}
	
The next proposition shows that the filtration $\{V_i(u)\}_{i=1}^{p(u)}$ is invariant under the action of $P^-_{D(u)}$.%
	
\begin{proposition}\label{V} 
For any $u\in\Omega$, $i\in\{1,\ldots,p(u)\}$ and $g\in P^-_{D(u)}$, the following statements hold:%
\begin{enumerate}
\item[(i)] $\Ad(g)V_i(u)=V_i(u)$;%
\item[(ii)] $\lim_{n\rightarrow+\infty}\frac{1}{n}\log|\Ad(\varphi_{n, u}(g))X|=\lambda_i(u)\iff X\in V_i(u)\setminus V_{i+1}(u)$;%
\item[(iii)] $\lim_{n\rightarrow+\infty}\frac{1}{n}\log\|\Ad(\varphi_{n, u}(g))_{|V_i(u)}\| = \lambda_i(u)$.%
\end{enumerate}
\end{proposition}
	
\begin{proof}
(i) Since $P^-_{D(u)} = k P^-_{\lambda^+(u)}k^{-1}$, where $k\in K$ satisfies $D(u)=\Ad(k)\lambda^+(u)$, it suffices to show that $\Ad(g)V^+_i(u)=V^+_i(u)$ for any $g\in P^-_{\lambda^+(u)}$. Moreover, since $P^-_{\lambda^+(u)} = N^-_{\lambda^+(u)}Z_{\lambda^+(u)}$ and $\Ad(g)\lambda^+(u) = \lambda^+(u)$ for any $g\in Z_{\lambda^+(u)}$, our work is reduced to showing the result for $g\in N_{\lambda^+(u)}$, observing that $\Ad(g)\lambda^+(u) = \lambda^+(u)$ implies $\Ad(g)U^+_j(u) = U^+_j(u)$ for any $j \in \{1,\ldots,p(u)\}$. By \eqref{decomposition} we have%
\begin{equation*}
  V_i^+(u) = \bigoplus_{j=i}^{p(u)}U_i^+(u) = \bigoplus_{\stackrel{\alpha\in\Pi}{\alpha(\lambda^+(u))\leq \lambda_i(u)}}\fg_{\alpha}.%
\end{equation*}
Moreover, for any $\beta\in \Pi^-\setminus\langle\Theta(\lambda^+(u))\rangle$ we have $\beta(\lambda^+(u))<0$ and so $(\alpha+\beta)(\lambda^+(u))<\alpha(\lambda^+(u))$ for any $\alpha\in\Pi$. Since $[\fg_{\alpha},\fg_{\beta}]\subset\fg_{\alpha+\beta}$ if $\alpha+\beta\in\Pi$, and $[\fg_{\alpha},\fg_{\beta}] = 0$ otherwise, we obtain%
\begin{equation*}%
	[\fn_{\lambda^+(u)}^-, V^+_i(u)]\subset \bigoplus_{j=i+1}^{p}U^+_j(u)\implies \Ad(g)V^+_i(u)\subset V^+_i(u),%
\end{equation*}
implying $\Ad(g)V^+_i(u) = V^+_i(u)$, using that $\Ad(g)$ is an isomorphism.%
		
(ii) Using right-invariance, by the previous discussion we have%
\begin{equation*}
  \lim_{n\rightarrow+\infty}\frac{1}{n}\log|\Ad(\varphi_{n,u}(g))X| = \lambda_i\iff \Ad(g)X\in V_i(u)\setminus V_{i+1}(u)%
\end{equation*}
which by the invariance in item (i) is equivalent to $X\in V_i(u)\setminus V_{i+1}(u)$.%
	
(iii) By the $P^-_{D(u)}$-invariance of $V_i(u)$, it is enough to show that%
\begin{equation*}
  \lim_{n\rightarrow+\infty}\frac{1}{n}\log\|\Ad(\varphi_{n,u}(e))_{|V_i(u)}\| = \lambda_i(u).%
\end{equation*}
Moreover, since%
\begin{equation*}
  \liminf_{n\rightarrow+\infty}\frac{1}{n}\log\|\Ad(\varphi_{n,u}(e))_{|V_i(u)}\| \geq \lambda_i(u)%
\end{equation*}
trivially follows from item (ii), our work is reduced to showing the opposite inequality. Write $X\in V_i(u)$ with $|X|=1$ as $X=\sum_{j=i}^{p(u)}P_j(u)X$. Then, by the uniformity of the convergence (see Theorem \ref{MET})%
\begin{equation*}
  \frac{1}{n}\log|\Ad(\varphi_{n, u}(e))P_j(u)X|\rightarrow\lambda_i(u), \;\;\;n\rightarrow+\infty%
\end{equation*}
for any $\ep>0$ there exists an integer $n_0>0$ such that%
\begin{equation*}
  |\Ad(\varphi_{n,u}(e))X| \leq \sum_{j=i}^{p(u)}\rme^{n(\lambda_j(u)+\varepsilon)}\|P_j(u)\| \leq p(u)\rme^{n(\lambda_i(u)+\varepsilon)}\|P_j(u)\|%
\end{equation*}
for any $n \geq n_0$ and $j\in\{i,\ldots,p(u)\}$. This easily implies the desired inequality%
\begin{equation*}
  \limsup_{n\rightarrow+\infty}\frac{1}{n}\log\|\Ad(\varphi_{n,u}(e))_{|V_i(u)}\| \leq \lambda_i(u),%
\end{equation*}
concluding the proof.
\end{proof}
	
Now we analyze the Lyapunov exponents of the systems $\Sigma_{\Theta}$. Since Lyapunov exponents do not change under time-discretization, we will only consider times $n\in\Z$.%
	
For any point $(u,x) \in \UC\times\F_{\Theta}$, let us denote by $\Lambda_{\Ly,\Theta}(u,x)$ the Lyapunov spectrum of $\Sigma_{\Theta}$ at $(u,x)$, i.e.,%
\begin{equation*}
  \Lambda_{\Ly,\Theta}(u,x) := \left\{\lambda\in\R\ :\ \lambda = \lim_{n\rightarrow+\infty}\frac{1}{n}\log|(\rmd\varphi_{n,u})_xv| \mbox{ for some } v\in T_x\F_{\Theta}\setminus\{0\}\right\}.%
\end{equation*}
Let us fix a $\theta_1$-invariant measure $\nu$ on $\UC$ and for any $w\in\WC$ consider%
\begin{equation*}
	\Lambda_{\Ly}(\st^{\nu}_{\Theta}(w)) := \bigcup_{(u,x)\in \st^{\nu}_{\Theta}(w)}\Lambda_{\Ly,\Theta}(u,x).%
\end{equation*}
The next result shows that on the above set we can recover all Lyapunov exponents from the $\fa$-Lyapunov exponents.%
	
\begin{theorem}\label{exponents}
For $w\in \WC$ and $(u,x) \in \st^{\nu}_{\Theta}(w)$ it holds that%
\begin{equation*}
  \Lambda_{\Ly,\Theta}(u,x) = \left\{\alpha(\lambda^+(u)): \;\alpha\in w\left(\Pi^-\setminus\langle\Theta\rangle\right)\right\}.%
\end{equation*}
\end{theorem}

\begin{proof} 
Let us consider as before the eigenspace $U_i(u)$ of $\ad(D(u))$ associated with $\lambda_i(u)$. Since $x\in\st_{\Theta}(D(u),w)$, we can write $x = g\cdot z$ with $g\in P^-_{D(u)}$ and $z\in\fix_{\Theta}(D(u),w)$. By Lemma \ref{auxiliar}, $\fg_z$ and $\fg_z^{\perp}$ are $\ad(D(u))$-invariant and therefore%
\begin{equation}\label{roots}
  \fg_z = \bigoplus_{j=1}^{p(u)}\left(U_j(u)\cap\fg_{z}\right) \;\;\;\mbox{ and }\;\;\; \fg_{z}^{\perp}=\bigoplus_{j=1}^{p(u)}\left(U_j(u)\cap\fg_{z}^{\perp}\right).%
\end{equation}
Let $i \in \{1,\ldots,p(u)\}$ be such that $U_i(u)\cap\fg_z^{\perp}\neq\{0\}$ and consider the vector $v = (\rmd[g])_zX(z)\in T_x\F_{\Theta}\setminus\{0\}$ with $X\in U_i(u)\cap\fg_z^{\perp}$. Since%
\begin{equation*}
  (\rmd\varphi_{n,u})_xv = (\rmd[\varphi_{n,u}(g)])_zX(z) = (\Ad(\varphi_{n, u}(g))X)(\varphi_{n, u}(x)),%
\end{equation*}
we obtain%
\begin{align*}
  |(\rmd\varphi_{n,u})_xv| &= |(\Ad(\varphi_{n,u}(g))X)(\varphi_{n,u}(x))|=|P_{\varphi_{n,u}(x)}(\Ad(\varphi_{n,u}(g))X)|\\
	                        &= \inf_{Y\in\fg_{\varphi_{n,u}(x)}}|\Ad(\varphi_{n,u}(g))X-Y|=\inf_{Z\in \fg_z}\left|\Ad(\varphi_{n,u}(g))(X-Z)\right|,%
\end{align*}
where we use that $\fg_{\varphi_{n,u}(x)} = \fg_{\varphi_{n,u}(g)\cdot z} = \Ad(\varphi_{n,u}(g))\fg_z$. Lemma \ref{liminf}(i) yields%
\begin{equation*}
  \lim_{n\rightarrow+\infty}\frac{1}{n}\log\left|(\rmd\varphi_{n,u})_{x}v\right| = \lim_{n\rightarrow+\infty}\frac{1}{n}\log \inf_{Z\in \fg_z}\left|\Ad(\varphi_{n,u}(g))(X-Z)\right| = \lambda_i(u).%
\end{equation*}
Now let $v\in T_x\F_{\Theta}\setminus\{0\}$. Since $T_x\F_{\Theta} = (\rmd[g])_zT_z\F_{\Theta} = (\rmd[g])_z\pi_z(\fg^{\perp}_z)$, by \eqref{roots} we obtain%
\begin{equation*}
  v = \sum_iv_i, \;\;\mbox{ where }\;\;v_i=(\rmd[g])_zX_i(z) \;\;\mbox{ with }\;\;0\neq X_i\in U_i(u)\cap\fg_{z}^{\perp}.%
\end{equation*}
Moreover, $v_i\neq 0$ iff $X_i\neq 0$, and by the above%
\begin{equation*}
  \lim_{n\rightarrow+\infty}\frac{1}{n}\log|(\rmd\varphi_{n,u})_xv_i| = \lambda_i(u).%
\end{equation*}
Since all $\lambda_i(u)$ are distinct, we obtain (see, e.g., \cite[Lem.~6.2.2]{CK2})%
\begin{equation*}
  \lim_{n\rightarrow+\infty}\frac{1}{n}\log|(\rmd\varphi_{n, u})_xv| = \max\{\lambda_i(u) : v_i\neq 0\},%
\end{equation*}
implying $\Lambda_{\Ly,\Theta}(u,x) = \left\{\lambda_i(u): \;  U_i(u)\cap\fg_z^{\perp}\neq\{0\}\right\}$. However, by considering $k\in K$ with $D(u) = \Ad(k)\lambda^+(u)$ and $z=k\cdot wb_{\Theta}$, we have%
\begin{equation*}
  U_i(u)\cap \fg_z^{\perp} = \Ad(k)\left(U^+_i(u)\cap w\fn_{\Theta}^-\right)
	                         = \Ad(k)\bigoplus_{\alpha\in w(\Pi^-\setminus\langle\Theta\rangle)\atop \alpha(\lambda^+(u)) = \lambda_i(u)}\fg_{\alpha},%
\end{equation*}
implying that $\Lambda_{\Ly,\Theta}(u,x)$ consists of all numbers $\alpha(\lambda^+(u))$ with $\alpha\in w(\Pi^-\setminus\langle\Theta\rangle)$, as stated.%
\end{proof}	

\subsection{Unstable, center and stable directions}
	
Let $\MC_{\Theta}(w)\subset\UC\times\F_{\Theta}$ be a Morse component of the control flow and consider the bundles $\EC^-_{\Theta,w}, \,\EC^0_{\Theta,w}$ and $\EC^+_{\Theta,w}$. For any $(u,x)\in\MC^{\nu}_{\Theta}(w)$ and $i = +,0,-$, let us define%
\begin{equation*}
  \Lambda^i_{\Ly,\Theta}(u,x) := \left\{\lambda \in\R: \; \lambda = \lim_{t\rightarrow+\infty}\frac{1}{t}\log|(\rmd\varphi_{t,u})_xv| \mbox{ for some } v\in \EC^i_{\Theta,w}(u,x)\setminus\{0\}\right\}%
\end{equation*}
and%
\begin{equation*}
  \Lambda^i_{\Ly}\left(\MC^{\nu}_{\Theta}(w)\right) := \bigcup_{(u,x)\in\MC^{\nu}_{\Theta}(w)}\Lambda_{\Ly,\Theta}^i(u,x).%
\end{equation*}
	
Let us also consider the subsets of roots given by%
\begin{equation*}
  \Pi^0_{\phi,\Theta,w} := \langle\Theta(\phi)\rangle\cap w(\Pi^-\setminus\langle\Theta\rangle) \;\;\mbox{ and }\;\;\Pi^{\pm}_{\phi, \Theta,w} := \Pi^{\pm}\setminus\langle\Theta(\phi)\rangle\cap w(\Pi^-\setminus\langle\Theta\rangle).%
\end{equation*}
	
\begin{theorem}\label{centralLy}
Let $\nu$ be a $\theta_1$-invariant measure on $\UC$. Then, for any $s\in\WC_{\Theta(\phi)}$ it holds that%
\begin{equation*}
  \Lambda^i_{\Ly,\Theta}(u,x) = \left\{\alpha(\lambda^+(u)): \alpha\in \Pi^i_{\phi,\Theta,sw}\right\} \;\;\mbox{ for all }\;\;(u, x)\in\MC_{\Theta}^{\nu}(w)\cap \st^{\nu}_{\Theta}(sw).%
\end{equation*}
Moreover,%
\begin{equation*}
  \Lambda^i_{\Ly}\left(\MC^{\nu}_{\Theta}(w)\right) \subset \bigcup_{\alpha\in \Pi^i_{\phi,\Theta,w}}\alpha(\Lambda_{\Ly}^+).%
\end{equation*}
\end{theorem}
	
\begin{proof} 
Let $(u,x)\in\MC_{\Theta}^{\nu}(w)$ and consider $s\in\WC_{\Theta(\phi)}$ such that $x\in\st_{\Theta}(D(u),sw) \cap \fix_{\Theta}(\mathsf{h}(u),w)$, which exists by Proposition \ref{MorseStable}. By Lemma \ref{technical}, we can write $x = g\cdot z$, where $g\in Z_{\mathsf{h}(u)}\cap N^-_{D(u)}$, $z\in\fix_{\Theta}(D(u),sw)$ and, for any $v\in\EC^i_{\Theta,w}(u,x)$ there is $X\in\fn^i_{\mathsf{h}(u)} \cap \fg_z^{\perp}$ such that $v = (\rmd[g])_zX(z)$, implying%
\begin{equation*}
  |(\rmd\varphi_{t,u})_xv| = |\Ad(\varphi_{t,u}(g))X(\varphi_{t,u}(x))|.%
\end{equation*}
Since $\fn^i_{\mathsf{h}(u)}\cap \fg_z^{\perp}$ is $\ad(D(u))$-invariant, we can proceed as in the proof of Theorem \ref{exponents} to conclude that%
\begin{equation}\label{eq_lyapspec}
  \Lambda^i_{\Ly,\Theta}(u,x) = \left\{\lambda_j(u) : U_j(u)\cap\bigl(\fn^i_{\mathsf{h}(u)}\cap \fg_z^{\perp}\bigr)\neq \{0\}\right\}.%
\end{equation}
Moreover, by considering $k\in K$ and $l\in K_{\lambda^+(u)}$ such that $\mathsf{h}(u)=\Ad(k)H_{\phi}$, $D(u)=\Ad(k)\lambda^+(u)$ and $z=k\cdot lswb_{\Theta}$, we obtain%
\begin{equation*}
  U_j(u)\cap(\fn^i_{\mathsf{h}(u)}\cap \fg_z^{\perp}) = \Ad(kl)\left(U^+_j(u)\cap\left(\fn^i_{H_{\phi}}\cap sw\fn_{\Theta}^-\right)\right) \subset \Ad(kl)\bigoplus_{\alpha \in \Pi^i_{\phi,\Theta,sw}}\fg_{\alpha},%
\end{equation*}
and therefore $\Lambda^i_{\Ly,\Theta}(u,x) \subset \{\alpha(\lambda^+(u)): \alpha\in \Pi^i_{\phi,\Theta,sw}\}$. On the other hand, since%
\begin{equation*}
  \fn^i_{\mathsf{h}(u)} \cap \fg_z^{\perp} = \Ad(kl)\left(\fn^i_{H_{\phi}} \cap sw\fn^-_{\Theta}\right) = \Ad(kl)\bigoplus_{\alpha \in \Pi^i_{\phi,\Theta,sw}}\fg_{\alpha},%
\end{equation*}
for any $\alpha \in \Pi^i_{\phi,\Theta,sw}$ we have $\fg_{\alpha} \subset U_j^+(u) \cap (\fn^i_{H_{\phi}} \cap sw\fn^-_{\Theta})$ for some $j \in \{1,\ldots,p(u)\}$, which shows that $\Lambda^i_{\Ly,\Theta}(u,x) \supset \{\alpha(\lambda^+(u)): \alpha\in \Pi^i_{\phi,\Theta,sw}\}$.%
		
For the second assertion, notice that by Proposition \ref{MorseStable} we have $\MC^{\nu}_{\Theta}(w)\subset \bigcup_{s\in\WC_{\Theta(\phi)}}\st^{\nu}_{\Theta}(sw)$, and hence%
\begin{align*}
  \Lambda^i_{\Ly}\left(\MC^{\nu}_{\Theta}(w)\right) &= \bigcup_{s\in\WC_{\Theta(\phi)}}\bigcup_{(u,x)\in\st^{\nu}_{\Theta}(sw)\cap\MC_{\Theta}(w)}\Lambda^i_{\Ly,\Theta}(u,x) = \bigcup_{\stackrel{s\in\WC_{\Theta(\phi)}}{u\in\Omega}}\left\{\beta\left(\lambda^+(u)\right): \beta\in \Pi^i_{\phi,\Theta,sw} \right\}\\
	&= \bigcup_{\stackrel{s\in\WC_{\Theta(\phi)}}{u\in\Omega}}\bigl\{\beta\left(\lambda^+(u)\right): \beta\in s\left(\Pi^i_{\phi,\Theta,w}\right) \bigr\} = \bigcup_{\stackrel{s\in\WC_{\Theta(\phi)}}{u\in\Omega}}\bigl\{\alpha\left(s^{-1}\lambda^+(u)\right): \alpha\in \Pi^i_{\phi,\Theta,w} \bigr\}\\
	&= \bigcup_{\alpha\in \Pi^i_{\phi,\Theta,w}}\alpha\left(\bigcup_{u\in\Omega}\WC_{\Theta(\phi)}\lambda^+(u)\right)
	 \subset \bigcup_{\alpha\in \Pi^i_{\phi,\Theta,w}}\alpha\left(\bigcup_{u\in\UC_{\reg}}\WC_{\Theta(\phi)}\lambda^+(u)\right)
	= \bigcup_{\alpha\in \Pi^i_{\phi,\Theta,w}}\alpha\left(\Lambda_{\Ly}^+\right).%
\end{align*}
The last identity follows from Proposition \ref{prop_spectralprops}(i). Hence, the proof is complete.
\end{proof}
	
The above relation between the Lyapunov exponents of the system and the $\fa$-Lyapunov exponents allows us to introduce the notion of the Morse spectrum on a Morse component as follows. For any Morse component $\MC_{\Theta}(w)\subset\UC\times\F_{\Theta}$ of the control flow and $i=+,0,-$,%
\begin{equation*}
  \Lambda^i_{\Mo}\left(\MC_{\Theta}(w)\right) := \bigcup_{\alpha \in \Pi^i_{\phi,\Theta,w}}\alpha\left(\Lambda_{\Mo}^+\right)%
\end{equation*}
is called the \emph{Morse spectrum} of $\Sigma_{\Theta}$ on $\MC_{\Theta}(w)$ in the direction of $\EC^i_{\Theta,w}$.%
	
The next result shows that $\Lambda^+_{\Mo}\left(\MC_{\Theta}(w)\right)$, $\Lambda^0_{\Mo}(\MC_{\Theta}(w))$ and $\Lambda^-_{\Mo}(\MC_{\Theta}(w))$ contain all the asymptotic information of the system provided by the Lyapunov exponents in direction of the stable, center and unstable bundles for the associated Morse component.%
		
\begin{theorem}\label{central}
For $i = +,0,-$ the following statements hold:%
\begin{itemize}
\item[(i)] $\Lambda^i_{\Mo}\left(\MC_{\Theta}(w)\right)$ is a compact subset of $\R$;%
\item[(ii)] $\Lambda^0_{\Mo}\left(\MC_{\Theta}(w)\right)$ is either empty or a symmetric interval containing the origin;%
\item[(iii)] The limits%
\begin{equation*}
  \lim_{n\rightarrow+\infty}\frac{1}{n}\max_{(u,x)\in\MC_{\Theta}(w)}\log\left\|(\rmd\varphi_{n,u})_{|\EC^i_{\Theta,w}(u,x)}\right\|\; \;\;\;\mbox{ and }	\;\;\;\;	\lim_{n\rightarrow+\infty}\frac{1}{n}\min_{(u,x)\in\MC_{\Theta}(w)}\log m\left((\rmd\varphi_{n,u})_{|\EC^i_{\Theta,w}(u,x)}\right)%
\end{equation*}
exist and	belong to $\Lambda^i_{\Mo}\left(\MC_{\Theta}(w)\right)$, provided that this set is nonempty.
\end{itemize}
\end{theorem}
	
\begin{proof} 
(i) Since $\Lambda_{\Mo}(\MC^+,\mathsf{a})$ is a compact subset of $\fa$ and%
\begin{equation*}
  \Lambda^i_{\Mo}\left(\MC_{\Theta}(w)\right) = \bigcup_{\alpha\in\Pi^i_{\phi,\Theta,w}}\alpha\left(\Lambda_{\Mo}^+\right)%
\end{equation*}
is a finite union, $\Lambda^i_{\Mo}(\MC_{\Theta}(w))$ is compact.%
		
(ii) By \cite[Cor.~8.6]{SMLS}, there exists $\lambda\in \Lambda_{\Mo}^+$ such that $\alpha(\lambda)=0$ for all $\alpha\in \langle\Theta(\phi)\rangle$. This implies that $\Lambda^0_{\Mo}(\MC_{\Theta}(w))$ is either empty or a finite union of connected sets each of which contains the origin, hence an interval containing the origin. For the symmetry let $\eta\in \Lambda^0_{\Mo}(\MC_{\Theta}(w))$ and consider $\alpha\in\Pi^0_{\phi,\Theta,w}$ and $\lambda\in \Lambda_{\Mo}^+$ such that $\eta=\alpha(\lambda)$. Since $\alpha$ belongs, in particular, to $\langle\Theta(\phi)\rangle$, we have $r_{\alpha}(\Lambda_{\Mo}^+) = \Lambda_{\Mo}^+$ (cf.~\cite[Cor.~8.5]{SMLS}), implying%
\begin{equation*}
  \Lambda^0_{\Mo}\left(\MC_{\Theta}(w)\right) \ni \alpha(r_{\alpha}(\lambda)) = (r^{-1}_{\alpha}(\alpha))(\lambda) = (-\alpha)(\lambda) = -\alpha(\lambda) = -\eta,%
\end{equation*}
where $r_{\alpha}$ is the orthogonal reflection at $\ker\alpha$. This concludes the proof of (ii).%
		
(iii) For $i=+,0,-$, let us define the subadditive cocycles $\sigma^i, \vartheta^i:\Z\times\MC_{\Theta}(w)\rightarrow\R$,%
\begin{equation*}
  \sigma^i_n(u,x) := \log\left\|(\rmd\varphi_{n,u})_{|\EC^i_{\Theta,w}(u,x)}\right\| \;\;\mbox{ and }\;\;\vartheta^i_n(u,x) := -\log m\left((\rmd\varphi_{n,u})_{|\EC^i_{\Theta,w}(u,x)}\right).%
\end{equation*}
Moreover, let us denote by $M^{\EC}_{\phi}$ the set of the $\phi_1$-invariant ergodic probability measures on the Borelians of $\MC_{\Theta}(w)$. The proof of item (iii) proceeds in two steps:%
		
\emph{Step 1}: We prove that for any $\rho \in M^{\EC}_{\phi}$ there exists a $\phi_1$-invariant set $\Omega_{\rho}$ of full measure contained in $\MC_{\Theta}(w)\cap\st^{\nu}_{\Theta}(sw)$ for some $s\in\WC_{\Theta(\phi)}$ such that%
\begin{align*}
  \lim_{n\rightarrow+\infty}\frac{1}{n}\sigma^i_n(u,x) &= \max\Lambda^i_{\Ly,\Theta}(u,x) = \max\Lambda^i_{\Ly,\Theta}(\Omega_{\rho}),\\
	\lim_{n\rightarrow+\infty}\frac{1}{n}\vartheta^i_n(u,x) &= -\min\Lambda^i_{\Ly,\Theta}(u,x) = -\min\Lambda^i_{\Ly,\Theta}(\Omega_{\rho})%
\end{align*}
for all $(u,x)\in\Omega_{\rho}$, where $\nu=f_*\rho$ is the push-forward of $\rho$ by $f = (\pi^{\Theta}_1)_{|\MC_{\Theta}(w)}$, the restriction to $\MC_{\Theta}(w)$ of the projection $\pi_1^{\Theta}$ onto the first component of $\UC\times\F_{\Theta}$.%
		
Let $\rho\in M^{\EC}_{\phi}$ and consider its push-forward $\nu = f_*\rho$ as above. Since $f$ projects $\MC_{\Theta}(w)$ onto $\UC$, $\nu$ is an ergodic probability measure on the Borelians of $\UC$ and by the MET (Theorem \ref{MET}) there exists a $\theta_1$-invariant subset $\Omega_{\nu}\subset\UC_{\reg}$ of full $\nu$-measure such that $\lambda^+(u) =: \lambda^+(\nu)$ is constant for all $u\in \Omega_{\nu}$. The set $\Omega_{\rho}:=f^{-1}(\Omega_{\nu})$ is a $\phi_1$-invariant set of full $\rho$-measure. Moreover, since $\Omega_{\rho}\subset\MC^{\nu}_{\Theta}(w)$ and $\MC_{\Theta}^{\nu}(w)$ is given by the disjoint union of the $\phi_1$-invariant measurable sets $\MC_{\Theta}(w)\cap\st_{\Theta}^{\nu}(sw)$ for $s\in\WC_{\Theta(\phi)}$ (see Proposition \ref{MorseStable}), the ergodicity of $\rho$ implies that $\Omega_{\rho}\subset\MC_{\Theta}(w)\cap\st^{\nu}_{\Theta}(sw)$ for some $s\in\WC_{\Theta(\phi)}$ (up to a set of measure zero). Therefore, by Theorem \ref{exponents} we have $\max\Lambda^i_{\Ly,\Theta}(\Omega_{\rho})=\max\Lambda^i_{\Ly,\Theta}(u,x)$ and $\min\Lambda^i_{\Ly,\Theta}(\Omega_{\rho})=\min\Lambda^i_{\Ly,\Theta}(u,x)$ for any $(u,x)\in\Omega_{\rho}$.%

Let then $(u,x)\in\Omega_{\rho}$ and write it as $x = g\cdot z$ with $g\in Z_{\mathsf{h(u)}}\cap N^-_{D(u)}$ and $z\in\fix_{\Theta}(D(u),sw)$ (see Proposition \ref{MorseStable} and Lemma \ref{technical}). If $j_0\in\{1,\ldots,p(u)\}$ is the smallest index such that $U_{j_0}(u)\cap (\fn^i_{\mathsf{h}(u)}\cap\fg_z^{\perp}) \neq \{0\}$, we have $\lambda_{j_0}(u) = \max\Lambda^i_{\Ly,\Theta}(u,x)$, see \eqref{eq_lyapspec}. By the definition of $\sigma^i_n$ and Proposition \ref{centralLy}, we have%
\begin{equation*}
  \liminf_{n\rightarrow\infty}\frac{1}{n}\sigma_n^i(u,x) \geq \lambda_{j_0}(u).%
\end{equation*}
On the other hand, by Lemma \ref{technical}, for any $V\in\EC^i_{\Theta,w}(u,x)$ there exists $X\in \fn^i_{\mathsf{h}(u)}\cap \fg_z^{\perp}$ such that $V = (\rmd[g])_zX(z)$, implying%
\begin{equation*}
  \left\|\left(\rmd\varphi_{n,u}\right)_{|\EC^i_{\Theta,w}(u,x)}\right\| \leq \left\|\Ad(\varphi_{n,u}(g))_{|\fn^i_{\mathsf{h}(u)}\cap \fg_z^{\perp}}\right\|\left\|(\rmd[g])_z^{-1}\right\|.%
\end{equation*}
By \eqref{eq_lyapspec} we have%
\begin{equation*}
  \lambda_{j_0}(u) = \max\left\{\lambda_j(u): \;\;U_j(u)\cap \left(\fn^i_{\mathsf{h}(u)}\cap\fg_z^{\perp}\right)\neq\{0\}\right\},%
\end{equation*}
and so (using Proposition \ref{K-invariant})%
\begin{equation*}
  \left\|\left(\rmd\varphi_{n,u}\right)_{|\EC^i_{\Theta,w}(u,x)}\right\| \leq \left\|\Ad(\varphi_{n,u}(g))_{|V_{j_0}(u)}\right\|\left\|\Ad\left(g^{-1}\right)\right\|,%
\end{equation*}
which by Proposition \ref{V} implies%
\begin{equation*}
  \limsup_{n\rightarrow+\infty}\frac{1}{n}\sigma^i_n(u,x) \leq \lim_{n\rightarrow+\infty}\frac{1}{n}\log\left\|\Ad(\varphi_{n,u}(g))_{|V_{j_0}(u)}\right\| = \lambda_{j_0}(u),%
\end{equation*}
and therefore%
\begin{equation*}
  \lim_{n\rightarrow+\infty}\frac{1}{n}\sigma^i_n(u,x) = \lambda_{j_0}(u).%
\end{equation*}
For $\vartheta_n$ we have by definition that%
\begin{equation*}
  \liminf_{n\rightarrow\infty}\frac{1}{n}\vartheta_n(u,x) \geq -\min \Lambda^i_{\Ly,\Theta}(u,x).%
\end{equation*}
Moreover, by Lemma \ref{technical} it holds that%
\begin{equation*}
  m\left((\rmd\varphi_{n,u})_{|\EC^i_{\Theta,w}(u,x)}\right) = \inf_{v\in \EC^i_{\Theta,w}(u,x), \,|v|=1}\left|(\rmd\varphi_{n,u})_xv\right| \geq \inf_{\stackrel{X\in\fg_z^{\perp} \cap \fn^i_{\mathsf{h}(u)}; c_1\leq |X|\leq c_2}{Y\in\fg_z}}\left|\Ad(\varphi_{n, u}(g))(X-Y)\right|,%
\end{equation*}
where $c_1 = \|(\rmd[g])_z\|^{-1}$ and $c_2 = \|(\rmd[g])_z^{-1}\|$ (see \eqref{ring}). By Lemma \ref{liminf}, for any $\ep>0$ there exist a constant $C>0$ and an integer $n_0>0$ satisfying%
\begin{equation*}
  \inf_{\stackrel{X\in \fg_z^{\perp} \cap \fn^i_{\mathsf{h}(u)}; c_1\leq|X|\leq c_2}{Y\in\fg_z}}\left|\Ad(\varphi_{n,u}(g))(X-Y)\right| \geq C\rme^{n\left(\min\Lambda^i_{\Ly,\Theta}(u,x)-\varepsilon\right)} \;\; \mbox{ for all }\;n\geq n_0.%
\end{equation*}
Therefore,%
\begin{equation*}
  \limsup_{n\rightarrow+\infty}\frac{1}{n}\vartheta^i_n(u,x) \leq -\min\Lambda^i_{\Ly,\Theta}(u,x) + \varepsilon,%
\end{equation*}
which concludes Step 1, since $\varepsilon>0$ was arbitrary.%

\emph{Step 2}: We prove that the limits%
\begin{equation*}
  \lim_{n\rightarrow+\infty}\frac{1}{n}\max_{(u,x)\in\MC_{\Theta}(w)}\log\left\|(\rmd\varphi_{n,u})_{|\EC^i_{\Theta,w}(u,x)}\right\|\; \;\;\;\mbox{ and }	\;\;\;\;	\lim_{n\rightarrow+\infty}\frac{1}{n}\min_{(u,x)\in\MC_{\Theta}(w)}\log m\left((\rmd\varphi_{n,u})_{|\EC^i_{\Theta,w}(u,x)}\right)%
\end{equation*}
exist and belong to $\Lambda^i_{\Mo}(\MC_{\Theta}(w))$.%

Since $\sigma^i_n$ is a continuous subadditive cocycle, by \cite[Thm.~A.3]{Ian} we have%
\begin{equation*}
  \lim_{n\rightarrow+\infty}\frac{1}{n}\max_{(u,x)\in\MC_{\Theta}(w)}\sigma^i_n(u,x) = \sup_{\nu\in M^{\EC}_{\phi}}\left\{\lim_{n\rightarrow+\infty}\frac{1}{n}\int\sigma^i_n\,\rmd\nu \right\}.%
\end{equation*}
Using Step 1 and the dominated convergence theorem (using that $\frac{1}{n}\sigma^i_n(u,x) \leq \max_{(v,y) \in \MC_{\Theta}(w)}\sigma^i_1(v,y)$ and $\frac{1}{n}\sigma^i_n(u,x) \geq -\frac{1}{n}\vartheta^i_n(u,x) \geq -\max_{(v,y)\in\MC_{\Theta}(w)}\vartheta^i_1(v,y)$ by subadditivity), we have%
\begin{equation*}
  \lim_{n\rightarrow+\infty}\frac{1}{n}\int\sigma^i_n\,\rmd\nu = \int\lim_{n\rightarrow+\infty}\frac{1}{n}\sigma^i_n\,\rmd\nu = \max\Lambda^i_{\Ly}(\Omega_{\rho}) \;\;\mbox{ for any }\;\rho\in M^{\EC}_{\phi}.%
\end{equation*}
Since $\Lambda^i_{\Ly}(\Omega_{\rho})\subset \Lambda^i_{\Mo}(\MC_{\Theta}(w))$ for any $\rho\in M^{\EC}_{\phi}$ and $\Lambda^i_{\Mo}(\MC_{\Theta}(w))$ is compact, we obtain%
\begin{equation*}
  \lim_{n\rightarrow+\infty}\frac{1}{n}\max_{(u,x)\in\MC_{\Theta}(w)}\log\left\|(\rmd\varphi_{n,u})_{|\EC^i_{\Theta,w}(u,x)}\right\| =\lim_{n\rightarrow+\infty}\frac{1}{n}\max_{(u,x)\in\MC_{\Theta}(w)}\sigma_n^i(u,x) \in \Lambda^i_{\Mo}\left(\MC_{\Theta}(w)\right).%
\end{equation*}
Analogously, using that $\vartheta^i$ is a subadditive cocycle, we get as above (using Step 1) that%
\begin{equation*}
  \lim_{n\rightarrow+\infty}\frac{1}{n}\max_{(u,x)\in\MC_{\Theta}(w)}\vartheta_n^i(u,x) \in -\Lambda^i_{\Mo}\left(\MC_{\Theta}(w)\right),%
\end{equation*}
and consequently%
\begin{equation*}
  \lim_{n\rightarrow+\infty}\frac{1}{n}\min_{(u,x)\in\MC_{\Theta}(w)}\log m\left((\rmd\varphi_{n,u})_{|\EC^i_{\Theta,w}(u,x)}\right) = -\lim_{n\rightarrow+\infty}\frac{1}{n}\max_{(u,x)\in\MC_{\Theta}(w)}\vartheta_n^i(u,x)\in \Lambda^i_{\Mo}\left(\MC_{\Theta}(w)\right),%
\end{equation*}
concluding the proof.
\end{proof}

As a direct corollary we have the following:%
	
\begin{corollary}
All limits of the form%
\begin{equation*}
  \lambda = \lim_{k\rightarrow+\infty}\frac{1}{n_k}\log\left\|(\rmd\varphi_{n_k,u_k})_{|\EC^i_{\Theta,w}(u_k,x_k)}\right\|,\quad (u_k,x_k) \in \MC_{\Theta}(w),\quad n_k \rightarrow + \infty,%
\end{equation*}
belong to $\Lambda^i_{\Mo}(\MC_{\Theta}(w))$.
\end{corollary}
	
\begin{proof}
Theorem \ref{central}(iii) yields%
\begin{align*}
  \lambda &= \lim_{k\rightarrow+\infty}\frac{1}{n_k}\log\left\|(\rmd\varphi_{n_k,u_k})_{|\EC_{\Theta,w}^i(u_k,x_k)}\right\| \leq \limsup_{n\rightarrow+\infty}\frac{1}{n}\max_{(u,x)\in\MC_{\Theta}(w)}\sigma^i_n(u,x)\\
	&= \lim_{n\rightarrow+\infty}\frac{1}{n}\max_{(u,x)\in\MC_{\Theta}(w)}\sigma_n^i(u,x)\in \Lambda^i_{\Mo}\left(\MC_{\Theta}(w)\right),%
\end{align*}
and consequently $\lambda\leq \max\Lambda^i_{\Mo}(\MC_{\Theta}(w))$. On the other hand, Theorem \ref{central}(iii) also implies%
\begin{align*}
  -\lambda &= -\lim_{k\rightarrow+\infty}\frac{1}{n_k}\log\left\|(\rmd\varphi_{n_k,u_k})_{|\EC_{\Theta,w}^i(u_k,x_k)}\right\| \leq \limsup_{n\rightarrow+\infty}\frac{1}{n}\max_{(u,x)\in \MC_{\Theta}(w)}\vartheta^i_{n}(u,x)\\
		&= \lim_{n\rightarrow+\infty}\frac{1}{n}\max_{(u, x)\in \MC_{\Theta}(w)}\vartheta^i_{n}(u, x)\in -\Lambda^i_{\Mo}\left(\MC_{\Theta}(w)\right).%
\end{align*}
Therefore, $\lambda\geq \min\Lambda^i_{\Mo}(\MC_{\Theta}(w))$ and hence $\lambda\in\Lambda^i_{\Mo}(\MC_{\Theta}(w))$, concluding the proof.
\end{proof}
	
The next result shows that the extremal points of the set $\Lambda_{\Mo}(\MC_{\Theta}(w))$ are Lyapunov exponents.%
	
\begin{proposition}\label{extremal}
For $i=+,0,-$, the extremal points of $\Lambda^i_{\Mo}(\MC_{\Theta}(w))$ are Lyapunov exponents.
\end{proposition}
	
\begin{proof}
We show the result only for $\max\Lambda^i_{\Mo}(\MC_{\Theta}(w))$, since the proof for $\min\Lambda^i_{\Mo}(\MC_{\Theta}(w))$ works analogously. By the definition of $\Lambda^i_{\Mo}(\MC_{\Theta}(w))$, we have%
\begin{equation*}
   \max\Lambda^i_{\Mo}\left(\MC_{\Theta}(w)\right) = \max_{\alpha\in\Pi^i_{\phi,\Theta,w}}\left\{\max\alpha\left(\Lambda_{\Mo}^+\right)\right\}.%
\end{equation*}
Moreover, since $\Lambda_{\Mo}^+$ is a convex subset of $\fa$ and $\alpha$ is linear, it holds that $\max\alpha(\Lambda_{\Mo}^+) =\alpha(\lambda)$ for some extremal point $\lambda\in \Lambda_{\Mo}^+$. By \cite[Cor.~5.2]{ASM1}, there exists $s\in\WC_{\Theta(\phi)}$ and an ergodic measure $\nu$ on $\UC$ such that $\lambda = s\lambda^+(\nu)$. Theorem \ref{centralLy} then implies%
\begin{equation*}
  \max\Lambda^i_{\Mo}\left(\MC_{\Theta}(w)\right) = \alpha(s\lambda^+(\nu)) \in \Lambda^i_{\Ly,\Theta}(u,x), \;\;\mbox{ where }\;(u,x)\in\MC_{\Theta}^{\nu}(w),%
\end{equation*}
concluding the proof.
\end{proof}
	
\begin{remark}\label{rem_posneg}
By the estimates \eqref{eq_su_estimates} and the above result it holds that%
\begin{equation*}
  \Lambda^-_{\Mo}\left(\MC_{\Theta}(w)\right) \subset (-\infty,-\zeta]\;\;\;\mbox{ and }\;\;\; \Lambda^+_{\Mo}\left(\MC_{\Theta}(w)\right) \subset [\zeta, \infty).%
\end{equation*}
\end{remark}

The above results show that $[\min\Lambda^i_{\Mo}(\MC_{\Theta}(w)),\max\Lambda^i_{\Mo}(\MC_{\Theta}(w))]$ is the smallest connected set containing all the Lyapunov exponents of the system at points in $\MC_{\Theta}(w)$ in the direction of the bundle $\EC^i_{\Theta,w}$. In particular, for the center bundle it holds that $\Lambda^0_{\Mo}(\MC_{\Theta}(w)) = [\min\Lambda^0_{\Mo}(\MC_{\Theta}(w)),\max\Lambda^0_{\Mo}(\MC_{\Theta}(w))]$. In the next section,
we will see that the knowledge of the Morse spectra is enough to characterize partial hyperbolicity of chain control sets.%

\section{Partial hyperbolicity}\label{sec_ph}

According to Definition \ref{def_phs}, the chain control set $E_{\Theta}(w) \subset \F_{\Theta}$ is \emph{partially hyperbolic} with the decomposition%
\begin{equation}\label{eq_splitting}
  T_x\F_{\Theta} = \EC^-_{\Theta,w}(u,x) \oplus \EC^0_{\Theta,w}(u,x) \oplus \EC^+_{\Theta,w}(u,x)%
\end{equation}
if there are constants $c>0$, $\lambda<0<\mu$ and $\lambda<\lambda'\leq\mu'<\mu$ such that for all $t\geq 0$ and $(u,x)\in\MC_{\Theta}(w)$ we have%
\begin{align*}
	\left|(\rmd\varphi_{t,u})_x v\right| &\leq c\rme^{\lambda t}|v| \mbox{\quad for all\ } v\in \EC_{\Theta,w}^-(u,x),\\
	c^{-1}\rme^{\lambda't}|v| \leq &\left|(\rmd\varphi_{t,u})_x v\right| \leq c\rme^{\mu' t}|v| \mbox{\quad for all\ } v\in \EC_{\Theta,w}^0(u, x),\\
	c^{-1}\rme^{\mu t}|v| &\leq \left|(\rmd\varphi_{t,u})_x v\right| \mbox{\quad for all\ } v\in \EC_{\Theta,w}^+(u,x),%
\end{align*}
and, additionally, at least one of the subbundles $\EC_{\Theta,w}^-$ and $\EC_{\Theta,w}^+$ is nontrivial.%


The next result shows that partial hyperbolicity on flag manifolds is equivalent to the assertion that the Lyapunov or Morse spectra corresponding to the three subbundles $\EC^i_{\Theta,w}$ have empty intersections.%
	
\begin{theorem}\label{Teo1}
Let $E_{\Theta}(w)\subset \F_{\Theta}$ be a chain control set of $\Sigma_{\Theta}$ such that $\dim\EC^+_{\Theta,w} + \dim\EC^-_{\Theta,w}\geq 1$. Then the following conditions are equivalent:%
\begin{itemize}
\item[(1)] $E_{\Theta}(w)$ is partially hyperbolic with the decomposition \eqref{eq_splitting}.%
\item[(2)] $\Lambda_{\Mo}^0(\MC_{\Theta}(w)) \cap \Lambda_{\Mo}^-(\MC_{\Theta}(w)) = \Lambda_{\Mo}^0(\MC_{\Theta}(w)) \cap \Lambda_{\Mo}^+(\MC_{\Theta}(w)) = \emptyset$.
\end{itemize}	
\end{theorem}
	
\begin{proof}
Assume that $E_{\Theta}(w)$ is partially hyperbolic. Then, for any $\theta_1$-invariant probability measure $\nu$ on $\UC$,%
\begin{equation*}
  \Lambda_{\Ly}^-\left(\MC^{\nu}_{\Theta}(w)\right) \subset (-\infty,\lambda], \;\;\Lambda_{\Ly}^0\left(\MC^{\nu}_{\Theta}(w)\right)\subset [\lambda', \mu']\;\;\mbox{ and }\;\; \Lambda_{\Ly}^+\left(\MC^{\nu}_{\Theta}(w)\right) \subset [\mu,+\infty),%
\end{equation*}
where the constants $\lambda < \lambda' \leq \mu' < \mu$ are given by the above definition.	By Proposition \ref{extremal}, we then have%
\begin{equation*}
  \Lambda_{\Mo}^-\left(\MC_{\Theta}(w)\right) \subset (-\infty, \lambda], \;\;\Lambda_{\Mo}^0\left(\MC_{\Theta}(w)\right) \subset [\lambda',\mu']\;\;\mbox{ and }\;\; \Lambda_{\Mo}^+\left(\MC_{\Theta}(w)\right) \subset [\mu, +\infty),%
\end{equation*}
and consequently%
\begin{equation*}
  \Lambda_{\Mo}^0\left(\MC_{\Theta}(w)\right) \cap \Lambda_{\Mo}^-\left(\MC_{\Theta}(w)\right) = \Lambda_{\Mo}^0\left(\MC_{\Theta}(w)\right) \cap \Lambda_{\Mo}^+\left(\MC_{\Theta}(w)\right) = \emptyset.%
\end{equation*}
Reciprocally, if $\Lambda_{\Mo}^0(\MC_{\Theta}(w)) \cap \Lambda_{\Mo}^-(\MC_{\Theta}(w)) = \Lambda_{\Mo}^0(\MC_{\Theta}(w)) \cap \Lambda_{\Mo}^+(\MC_{\Theta}(w)) = \emptyset$, using the observation in Remark \ref{rem_posneg}, we can choose $\lambda < 0 < \mu$ and $\lambda'\leq \mu'$ such that%
\begin{equation*}
  \max\Lambda_{\Mo}^-\left(\MC_{\Theta}(w)\right) < \lambda < \lambda' < \min\Lambda_{\Mo}^0\left(\MC_{\Theta}(w)\right) \leq \max\Lambda_{\Mo}^0\left(\MC_{\Theta}(w)\right) < \mu' < \mu < \min\Lambda_{\Mo}^+\left(\MC_{\Theta}(w)\right).%
\end{equation*}
By Theorem \ref{central}(iii), there exists an integer $t_0>0$ such that for all $(u,x) \in \MC_{\Theta}(w)$ and $t \geq t_0$ it holds that%
\begin{align*}
  \left\|(\rmd\varphi_{t,u})_{|\EC^-_{\Theta,w}(u,x)}\right\| < \rme^{\lambda t} &\mbox{\quad and\quad} m\left((\rmd\varphi_{t,u})_{|\EC^+_{\Theta,w}(u,x)}\right) > \rme^{\mu t},\\
  \left\|(\rmd\varphi_{t,u})_{|\EC^0_{\Theta,w}(u,x)}\right\| < \rme^{\mu't} &\mbox{\quad and\quad} m\left((\rmd\varphi_{t,u})_{|\EC^0_{\Theta,w}(u,x)}\right) > \rme^{\lambda't}.%
\end{align*}
By considering%
\begin{align*}
  c_1 := \max\left\{1,\max_{\stackrel{0\leq t\leq t_0}{(u,x)\in\MC_{\Theta}(w)}}\left\|(\rmd\varphi_{t, u})_{|\EC^-_{\Theta,w}(u,x)}\right\|\rme^{-\lambda t}, \max_{\stackrel{0\leq t\leq t_0}{(u,x)\in\MC_{\Theta}(w)}}\left\|(\rmd\varphi_{t,u})_{|\EC^0_{\Theta,w}(u,x)}\right\|\rme^{-\mu' t}\right\},\\
	c_2 := \min\left\{1,\min_{\stackrel{0\leq t\leq t_0}{(u,x)\in\MC_{\Theta}(w)}}m\left((\rmd\varphi_{t,u})_{|\EC^+_{\Theta,w}(u,x)}\right)\rme^{-\mu t}, \min_{\stackrel{0\leq t\leq t_0}{(u,x)\in\MC_{\Theta}(w)}}m\left((\rmd\varphi_{t, u})_{|\EC^0_{\Theta,w}(u,x)}\right)\rme^{-\lambda't}\right\}%
\end{align*}
and $c := \max\{c_1,c_2^{-1}\}$, for all $(u,x)\in \MC_{\Theta}(w)$ and $t\geq 0$ we have%
\begin{align*}
  \left\|(\rmd\varphi_{t,u})_{|\EC^-_{\Theta,w}(u,x)}\right\| &\leq c\rme^{\lambda t} \mbox{\quad and\quad} m\left((\rmd\varphi_{t,u})_{|\EC^+_{\Theta,w}(u,x)}\right) \geq c^{-1}\rme^{\mu t},\\
	\left\|(\rmd\varphi_{t,u})_{|\EC^0_{\Theta,w}(u,x)}\right\| &\leq c\rme^{\mu't} \mbox{\quad and\quad} m\left((\rmd\varphi_{t,u})_{|\EC^0_{\Theta,w}(u,x)}\right) \geq c^{-1}\rme^{\lambda't},%
\end{align*}
implying that $E_{\Theta}(w)$ is partially hyperbolic.%
\end{proof}
	
\begin{remark}
Since $\Lambda_{\Mo}^0(\MC_{\Theta}(w))$ is symmetric, partial hyperbolicity is also equivalent to%
\begin{equation*}
	\max\Lambda_{\Mo}^0\left(\MC_{\Theta}(w)\right) \leq \min\left\{\min\Lambda_{\Mo}^+\left(\MC_{\Theta}(w)\right),-\max\Lambda_{\Mo}^-\left(\MC_{\Theta}(w)\right) \right\}.%
\end{equation*}
\end{remark}
	
\begin{remark}
The uniformly hyperbolic case (without center bundle) has been characterized in \cite{DS2} by the condition $\langle \Theta(\phi)\rangle \subset w\langle \Theta \rangle$. In this case, $\Pi^0_{\phi,\Theta,w} = \emptyset$, and consequently $\Lambda^0_{\Mo}(\MC_{\Theta}(w)) = \emptyset$.%
\end{remark}

\section{Examples}\label{sec_ex}

\subsection{Right-invariant systems on the flag manifolds of $G = \Sl(3,\R)$}

Let $G = \Sl(3,\R)$ and $\fg = \sl(3,\R)$ with the following canonical choices:%
\begin{itemize}
\item $\fa=\{\diag(a_1,a_2,a_3): \;\;a_i\in\R,\ a_1+a_2+a_3=0\}$;%
\item $\Pi=\{\alpha_{i,j}: i,j\in\{1,2,3\} \;\mbox{ with }\; i\neq j\}$, where $\fa \ni H \mapsto \alpha_{i,j}(H) := a_i-a_j\in\R$;%
\item $\Pi^+=\{\alpha_{12}, \alpha_{13}, \alpha_{23}\}$ and $\Pi^-=-\Pi^+$;%
\item $\WC$ is the permutation group $S_3$ which acts on the matrices in $\fa$ by permuting the entries on the diagonal.%
\item $\F=\{V_1\subset V_2\subset\R^3\}$ where $V_i$ are vector subspaces of $\R^3$ with $\dim V_i=i$, $i=1,2$;%
\item $\F_{\{\alpha_{23}\}} = \mathbb{RP}^2$ the two-dimensional real projective space;%
\item $\F_{\{\alpha_{12}\}} = \mathrm{Gr}_2(\R^3)$ the Grassmannian of the two-dimensional vector subspaces of $\R^3$.
\end{itemize}

\begin{figure}[hb]
\begin{center}
\includegraphics[scale=1.50]{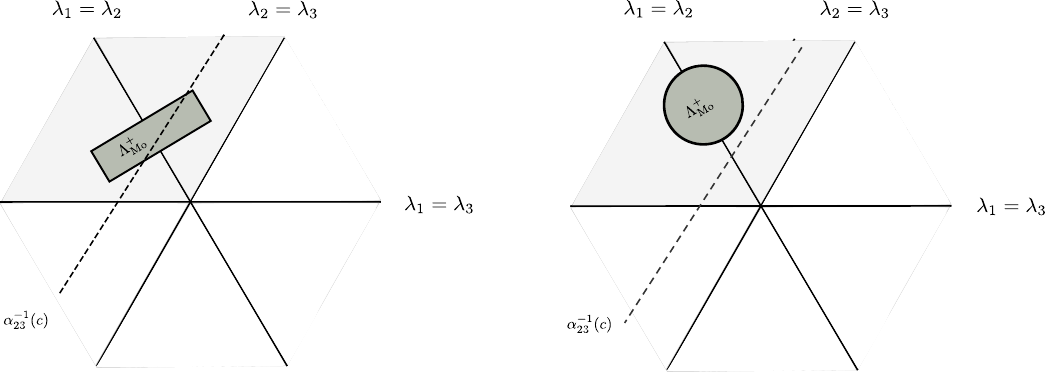}
\end{center}
\caption{Morse spectra}
\label{fig1}
\end{figure}

Let $\Sigma$ be an invariant system on the maximal flag manifold $\F$. We have four possibilities for $H_{\phi}$, namely:%

\subsubsection*{$\bullet \;\;H_{\phi}=\diag(a_1>a_2>a_3)$}

In this case, $\Theta(\phi)=\emptyset$ and by \cite[Thm.~4.6]{DS2} any chain control set $E_{\Theta}(w)$ of the induced system $\Sigma_{\Theta}$ on $\F_{\Theta}$ is uniformly hyperbolic without center bundle for any choice of $\Theta$, i.e., $\EC^0_{\Theta,w}$ is trivial.%
	
\subsubsection*{$\bullet\;\;H_{\phi}=\mathrm{diag}(a_1=a_2=a_3)=0$}%

In this case, $\Theta(\phi) = \Sigma$. By \cite[Prop.~8.7]{SMLS} the control flow on $\UC\times\F_{\Theta}$ is chain transitive, implying that the system $\Sigma_{\Theta}$ is also chain transitive on $\F_{\Theta}$. In particular, $\EC^{\pm}_{\Theta,w}$ are trivial, hence no chain control set is partially hyperbolic.%

\subsubsection*{$\bullet\;\;H_{\phi}=\mathrm{diag}(a_1=a_2>a_3)$}

In this case, $\Theta(\phi)=\{\alpha_{12}\}$ and a simple calculation yields:%

{\bf The induced system on $\mathbf{\F}$:} It admits three Morse sets%
\begin{equation*}
  \MC(w_1) = \MC(w_{12}), \;\;\MC(w_{23}) = \MC(w_{123}) \;\;\mbox{ and }\;\;\MC(w_{13})=\MC(w_{132}),%
\end{equation*}
where the subscript in the element of $\WC$ corresponds to the permutations in the diagonal of the elements in $\fa$ and%
\begin{equation*}
	\begin{array}{lllll}
	\Pi^+_{\phi,w_1} = \emptyset, & \hspace{1cm} &\Pi^0_{\phi, w_1}=\{-\alpha_{12}\} & \hspace{1cm} & \Pi^-_{\phi, w_1}=\{-\alpha_{13}, -\alpha_{23}\}\\
	\Pi^+_{\phi,w_{23}} = \{\alpha_{23}\} &\hspace{1cm} & \Pi^0_{\phi, w_{23}}=\{-\alpha_{12}\} & \hspace{1cm} & \Pi^-_{\phi, w_{23}}=\{-\alpha_{13}\}\\
	\Pi^+_{\phi,w_{13}} = \{\alpha_{13}, \alpha_{23}\} & \hspace{1cm}&\Pi^0_{\phi, w_{13}}=\{\alpha_{12}\} & \hspace{1cm} & \Pi^-_{\phi, w_{13}}=\emptyset
	\end{array}
\end{equation*}
implying%
\begin{align*}
	\Lambda_{\Mo}^+\left(\MC_{\{\alpha_{23}\}}(w_1)\right)&=\emptyset \allowdisplaybreaks\\
	\Lambda_{\Mo}^0\left(\MC_{\{\alpha_{23}\}}(w_1)\right)&=\alpha_{12}\left(\Lambda_{\Mo}^+\right) \allowdisplaybreaks\\
	\Lambda_{\Mo}^-\left(\MC_{\{\alpha_{23}\}}(w_1)\right)&=-\left(\alpha_{13}\left(\Lambda_{\Mo}^+\right)\cup\alpha_{23}\left(\Lambda_{\Mo}^+\right)\right) \allowdisplaybreaks\\
	\Lambda_{\Mo}^+\left(\MC_{\{\alpha_{23}\}}(w_{23})\right)&=\alpha_{23}\left(\Lambda_{\Mo}^+\right) \allowdisplaybreaks\\
	\Lambda_{\Mo}^0\left(\MC_{\{\alpha_{23}\}}(w_{23})\right)&=\alpha_{12}\left(\Lambda_{\Mo}^+\right) \allowdisplaybreaks\\
	\Lambda_{\Mo}^-\left(\MC_{\{\alpha_{23}\}}(w_{23})\right)&=-\alpha_{13}\left(\Lambda_{\Mo}^+\right) \allowdisplaybreaks\\
	\Lambda_{\Mo}^-\left(\MC_{\{\alpha_{23}\}}(w_{13})\right)&=\emptyset \allowdisplaybreaks\\
	\Lambda_{\Mo}^0\left(\MC_{\{\alpha_{23}\}}(w_{13})\right)&=\alpha_{12}\left(\Lambda_{\Mo}^+\right) \allowdisplaybreaks\\
	\Lambda_{\Mo}^+\left(\MC_{\{\alpha_{23}\}}(w_{13})\right)&=\alpha_{13}\left(\Lambda_{\Mo}^+\right)\cup\alpha_{23}\left(\Lambda_{\Mo}^+\right)%
\end{align*}

By Theorem \ref{Teo1}, the chain control sets $E(w_1)$, $E(w_{13})$ and $E(w_{23})$ are partially hyperbolic iff%
\begin{equation*}
  \min\left\{\min_{\lambda\in \Lambda_{\Mo}^+}\alpha_{23}(\lambda), \min_{\lambda\in \Lambda_{\Mo}^+}\alpha_{13}(\lambda)\right\}>\max_{\lambda\in \Lambda_{\Mo}^+}\alpha_{12}(\lambda).%
\end{equation*}
	
Since $\Lambda_{\Mo}^+$ is convex and invariant by $w_{12}$ and $w_{12}(\ker\alpha_{23})=\ker\alpha_{13}$, we have $\min_{\lambda\in \Lambda_{\Mo}^+}\alpha_{13}(\lambda)=\min_{\lambda\in \Lambda_{\Mo}^+}\alpha_{23}(\lambda)$ (see Fig.~\ref{fig1}), and consequently $E(w_1),E(w_{13})$ and $E(w_{23})$ are partially hyperbolic iff%
\begin{equation*}
  \min_{\lambda\in \Lambda_{\Mo}^+}\alpha_{23}(\lambda)>\max_{\lambda\in \Lambda_{\Mo}^+}\alpha_{12}(\lambda)%
\end{equation*}
iff $\Lambda_{\Mo}^+$ is contained in the upper half space determined by the line $\alpha_{23}^{-1}(c)$, where $c=\max_{\lambda\in \Lambda_{\Mo}^+}\alpha_{12}(\lambda)$ (see Fig.~\ref{fig1}).%
	
{\bf The induced system on $\mathbf{\F_{\alpha_{23}}=\mathbb{RP}^2}$:} For this setup, we have the Morse components%
\begin{align*}
  \MC_{\{\alpha_{23}\}}(w_1)&=\MC_{\{\alpha_{23}\}}(w_{12})=\MC_{\{\alpha_{23}\}}(w_{23})=\MC_{\{\alpha_{23}\}}(w_{123})\\ 
	&\mbox{ and }
	\MC_{\{\alpha_{23}\}}(w_{13})=\MC_{\{\alpha_{23}\}}(w_{132})%
\end{align*}
which yields%
\begin{equation*}
	\begin{array}{lllll}
	\Pi^+_{\phi, \{\alpha_{23}\}, w_1}=\emptyset, & \hspace{1cm} &\Pi^0_{\phi, \{\alpha_{23}\}, w_1}=\{-\alpha_{12}\} & \hspace{1cm} & \Pi^-_{\phi, \{\alpha_{23}\}, w_1}=\{-\alpha_{13}\}\\
	\Pi^+_{\phi, \{\alpha_{23}\}, w_{13}}=\{\alpha_{13}, \alpha_{23}\} &\hspace{1cm} & \Pi^0_{\phi, \{\alpha_{23}\}, w_{13}}=\emptyset & \hspace{1cm} & \Pi^-_{\phi, \{\alpha_{23}\}, w_{13}}=\emptyset
	\end{array}
\end{equation*}
and%
\begin{align*}
	\Lambda_{\Mo}^+\left(\MC_{\{\alpha_{23}\}}(w_1)\right) &=\emptyset\\
	\Lambda_{\Mo}^0\left(\MC_{\{\alpha_{23}\}}(w_1)\right) &= \alpha_{12}\left(\Lambda_{\Mo}^+\right)\\
	\Lambda_{\Mo}^-\left(\MC_{\{\alpha_{23}\}}(w_1)\right) &= -\alpha_{13}\left(\Lambda_{\Mo}^+\right)\\
	\Lambda_{\Mo}^+\left(\MC_{\{\alpha_{23}\}}(w_{13})\right) &= \alpha_{13}\left(\Lambda_{\Mo}^+\right)\cup\alpha_{23}\left(\Lambda_{\Mo}^+\right)\\
	\Lambda_{\Mo}^0\left(\MC_{\{\alpha_{23}\}}(w_{13})\right) &= \emptyset\\
	\Lambda_{\Mo}^-\left(\MC_{\{\alpha_{23}\}}(w_{13})\right) &= \emptyset&
\end{align*}
By \cite[Thm.~4.6]{DS2}, $E_{\{\alpha_{23}\}}(w_{13})$ is uniformly hyperbolic without center bundle and by Theorem \ref{Teo1} above $E_{\{\alpha_{23}\}}(w_{1})$ is partially hyperbolic iff $\min_{\lambda\in \Lambda_{\Mo}^+}\alpha_{13}(\lambda)>\max_{\lambda\in \Lambda_{\Mo}^+}\alpha_{12}(\lambda)$.%
		
{\bf The induced system on $\mathbf{\F_{\alpha_{12}}=\mathrm{Gr}_2(\R^3)}$:} In this case, we have the Morse components%
\begin{equation*}
  \MC_{\{\alpha_{12}\}}(w_1)=\MC_{\{\alpha_{12}\}}(w_{12}) \hspace{1cm}\mbox{ and }$$ $$\MC_{\{\alpha_{12}\}}(w_{13})=\MC_{\{\alpha_{12}\}}(w_{23})=\MC_{\{\alpha_{12}\}}(w_{132})=\MC_{\{\alpha_{12}\}}(w_{123})%
\end{equation*}
which yields%
\begin{equation*}
	\begin{array}{lllll}
	\Pi^+_{\phi, \{\alpha_{12}\}, w_1}=\emptyset, & \hspace{1cm} &\Pi^0_{\phi, \{\alpha_{12}\}, w_1}=\emptyset & \hspace{1cm} & \Pi^-_{\phi, \{\alpha_{12}\}, w_1}=\{-\alpha_{13}, -\alpha_{23}\}\\
	\Pi^+_{\phi, \{\alpha_{12}\}, w_{13}}=\{\alpha_{13}\} &\hspace{1cm} & \Pi^0_{\phi, \{\alpha_{12}\}, w_{13}}=\{\alpha_{12}\} & \hspace{1cm} & \Pi^-_{\phi, \{\alpha_{12}\}, w_{13}}=\emptyset\\
	\end{array}
\end{equation*}
and%
\begin{align*}
	\Lambda_{\Mo}^+\left(\MC_{\{\alpha_{12}\}}(w_1)\right) &= \emptyset\\
	\Lambda_{\Mo}^0\left(\MC_{\{\alpha_{12}\}}(w_1)\right) &= \emptyset\\
	\Lambda_{\Mo}^-\left(\MC_{\{\alpha_{12}\}}(w_1)\right) &= -\left(\alpha_{13}\left(\Lambda_{\Mo}^+\right)\cup\alpha_{23}\left(\Lambda_{\Mo}^+\right)\right)\\
	\Lambda_{\Mo}^+\left(\MC_{\{\alpha_{12}\}}(w_{13})\right) &= \alpha_{13}\left(\Lambda_{\Mo}^+\right)\\
	\Lambda_{\Mo}^0\left(\MC_{\{\alpha_{12}\}}(w_{13})\right) &= \alpha_{12}\left(\Lambda_{\Mo}^+\right)\\
	\Lambda_{\Mo}^-\left(\MC_{\{\alpha_{22}\}}(w_{12})\right) &= \emptyset%
\end{align*}
Therefore, $E_{\{\alpha_{12}\}}(w_1)$ is uniformly hyperbolic without center bundle and $E_{\{\alpha_{12}\}}(w_{13})$ is partially hyperbolic iff $\min_{\lambda\in \Lambda_{\Mo}^+}\alpha_{23}(\lambda)>\max_{\lambda\in \Lambda_{\Mo}^+}\alpha_{12}(\lambda)$.%
	
{\bf Summarizing:} For any induced invariant system on $\F_{\Theta}$ such that $\Theta(\phi)=\{\alpha_{12}\}$, any chain control set is partially hyperbolic iff $\Lambda_{\Mo}^+$ is in the upper half space determined by the line $\alpha_{23}^{-1}(c)$, where $c=\max_{\lambda\in\Lambda_{\Mo}^+}\alpha_{12}(\lambda)$.%

\subsubsection*{$\bullet\;\;H_{\phi} = \diag(a_1>a_2=a_3)$:} This case is analogous to the preceding one. An analogous analysis shows that for any induced invariant system on $\F_{\Theta}$ such that $\Theta(\phi) = \{\alpha_{23}\}$, any chain control set is partially hyperbolic iff $\Lambda^+_{\Mo}$ is contained in the upper half space determined by the line $\alpha_{12}^{-1}(c)$, where $c = \max_{\lambda\in\Lambda^+_{\Mo}}\alpha_{23}(\lambda)$.%

\subsection{An example on the 2-torus}

Let us consider $X,Y\in\sl(2)$ with $\det[X,Y]\neq 0$. It is not hard to see that this implies $\sl(2) = \mathrm{span}\{X,Y,[X,Y]\}$. For any $\rho>0$ let%
\begin{equation*}
  \UC_{\rho} := \{u\in L^{\infty}(\R, \R)\ :\ u(t) \in [-\rho,\rho] \;\mbox{ a.e.}\}%
\end{equation*}
and consider the bilinear system on $\R^2$ given by%
\begin{equation}\label{bilinear}
  \dot{x}(t) = (X + u(t)Y)x(t),\quad u\in\UC^{\rho}.%
\end{equation}
Such systems factor naturally to the projective line $\mathbb{P}$ and we denote these induced systems by $\Sigma_{\mathbb{P}}^{\rho}$.
The condition that $X$ and $Y$ generate $\sl(2)$ together with the compactness of $\mathbb{P}$ implies, in particular, that for $\rho>0$ small enough, $(u,x) \in \UC_{\rho} \tm \mathbb{P}$ is an inner pair, i.e., there is $\tau>0$ with $\varphi^{\rho}_{\tau,u}(x)\in\inner\OC^{+,\rho}(x)$ (see \cite[Prop.~4.5.19]{CKl}), where $\OC^{+,\rho}(x)$ is the positive orbit of $x$ using controls in $\UC_{\rho}$. By \cite[Thm.~6.1.3(iv)]{CKl}, the map $\rho\mapsto \Lambda_{\Mo}(E^{\rho})$ is continuous (in particular) at $\rho=0$, where $\Lambda_{\Mo}(E^{\rho})$ is the Morse spectrum of the bilinear system \eqref{bilinear} over the chain control set $E^{\rho}$ of $\Sigma_{\mathbb{P}}^{\rho}$.%

The system $\Sigma_{\mathbb{P}}^{\rho}$ coincides with the invariant system on the maximal flag manifold $\mathbb{P}$ of $\sl(2)$ induced by%
\begin{equation*}
  \dot{g}(t) = X(g(t)) + u(t)Y(g(t)),%
\end{equation*}
where $u\in\UC_{\rho}$. If we denote by $\varphi^{\mathbb{P},\rho}$ and $\varphi^{\rho}$ the transition maps of the systems on $\mathbb{P}$ and on $\R^2$, respectively, a simple calculation shows that%
\begin{equation}\label{eq}
  \log |(\rmd\varphi^{\mathbb{P}}_{t,u})_xZ(x)| = D - 2\log|\varphi_{t,u}(ke_1)|,%
\end{equation}
where $x = k\cdot b_0$, $k\in \SO(2)$, $Z\in\sl(2)$ and the number $D$ only depends on $k$ and $Z$. Therefore, the Lyapunov exponents of the system $\Sigma_{\mathbb{P}}^{\rho}$ can be recovered from the Lyapunov exponents of the bilinear system on $\R^2$ and vice-versa.%

On the other hand, by \cite[Thm.~5.2]{SMVA}, the induced system $\Sigma_{\mathbb{P}}^{\rho}$ satisfies:%
\begin{itemize}
\item[1.] If $\det X\geq 0$, then $\Sigma_{\mathbb{P}}^{\rho}$ is controllable for any $\rho>0$;%
\item[2.] If $\det X<0$ and $\det [X,Y]<0$, then $\Sigma_{\mathbb{P}}^{\rho}$ is controllable iff there is $u_0\in\Omega_{\rho}$ such that $X+u_0Y$ has only purely imaginary eigenvalues;%
\item[3.] If $\det X<0$ and $\det [X,Y]>0$, then $\Sigma_{\mathbb{P}}^{\rho}$ is uncontrollable for any $\rho>0$.%
\end{itemize}
Moreover, in the case when $X+uY$ has a pair of nonzero real eigenvalues for any $u\in\Omega_{\rho}$, the system $\Sigma_{\mathbb{P}}$ admits two control sets.%

Let us then consider%
\begin{equation*}
  X_1 = \left(\begin{array}{cc}
1 & 0 \\ 0 & -1
\end{array}\right), \;\;X_2 = \left(\begin{array}{cc}
0 & 1 \\ 0 & 0
\end{array}\right) \;\;\mbox{ and } Y_i \in \sl(2) \mbox{ satisfying } \det[X_i,Y_i]> 0,\ i = 1,2.%
\end{equation*}
By the above, for $\rho>0$ small enough, the control-affine systems $\Sigma_{\mathbb{P}}^{\rho,1}$ and $\Sigma_{\mathbb{P}}^{\rho,2}$ induced on $\mathbb{P}$, respectively, by the bilinear systems on $\R^2$ given by%
\begin{equation*}
  \dot{x}(t) = (X_1+u_1(t) Y_1)x(t)\;\mbox{ and }\; \dot{x}(t) = (X_2+u_2(t) Y_2)x(t) \;\;\mbox{ with }\;\;u_1,u_2\in \UC_{\rho}%
\end{equation*}
satisfy:%
\begin{itemize}
\item[1.] $\Sigma_{\mathbb{P}}^{\rho,1}$ admits two disjoint control sets whose closures are chain control sets;
\item[2.] $\Sigma_{\mathbb{P}}^{\rho,2}$ is controllable.
\end{itemize}
Moreover, since $X_2$ is nilpotent, the continuity of the spectrum of the bilinear system together with equation \eqref{eq} implies that for any $\ep>0$ there is $\rho>0$ small enough such that the Lyapunov exponents of $\Sigma_{\mathbb{P}}^{\rho,2}$ are contained in $(-\ep,\ep)$. On the other hand, if we denote by $E^{\rho,+}$ and $E^{\rho,-}$ the chain control sets of $\Sigma_{\mathbb{P}}^{\rho,1}$, by continuity, the Lyapunov exponents of $\Sigma_{\mathbb{P}}^{\rho,1}$ on $E^{\rho,+}$ are contained in $(-2-\ep,-2+\ep)$ and on $E^{\rho,-}$ in $(2-\ep,2+\ep)$, respectively.%

Let us now consider the semisimple Lie group $G = \Sl(2) \tm \Sl(2)$ with Lie algebra $\fg = \sl(2) \tm \sl(2)$. If $X=(X_1,X_2)$, $Y=(Y_1,0)$, $Z=(0,Y_2)\in\fg$, the right-invariant system on $G$ given by%
\begin{equation*}
  \dot{g}(t) = X(g(t)) + u_1(t)Y(g(t)) + u_2(t)Z(g(t)), \quad u_1,u_2\in \UC_{\rho}%
\end{equation*}
factors to the system $\Sigma^{\rho}_{\F}$ on the maximal flag manifold $\F = \mathbb{P} \tm \mathbb{P}$ and satisfies $\Sigma^{\rho}_{\F} = \Sigma_{\mathbb{P}}^{\rho,1} \tm \Sigma_{\mathbb{P}}^{\rho,2}$.%

Therefore, $\Sigma^{\rho}_{\F}$ admits two chain control sets $E^{\rho}(1) = E^{\rho,+} \tm \mathbb{P}$ and $E(w_0) = E^{\rho,-} \tm \mathbb{P}$ (see Fig.~\ref{fig2}). In particular, the flag type of the control flow $\phi^{\rho}$ of $\Sigma_{\F}^{\rho}$ is given by $\Theta(\phi^{\rho}) = \{(0,\alpha)\}$\footnote{Here, $(0,\alpha)$ stands for the linear functional given by the composition of the root $\alpha$ of $\sl(2)$ with the projection in the second coordinate.}, implying%
\begin{equation*}\begin{array}{lll}
  \Pi^+_{\phi^{\rho}, 1}=\emptyset, &\Pi^-_{\phi^{\rho}, 1}=\{(-\alpha, 0)\} & \Pi^0_{\phi^{\rho}, 1}=\{(0, -\alpha)\}\\
  \Pi^+_{\phi^{\rho}, w_0}=\{(-\alpha, 0)\}, & \Pi^-_{\phi^{\rho}, w_0}=\emptyset & \Pi^0_{\phi^{\rho}, w_0}=\{(0, \alpha)\}
\end{array}
\end{equation*}
and consequently that $E^{\rho}(1)$ and $E^{\rho}(2)$ are not uniformly hyperbolic.%

On the other hand, by considering%
\begin{equation*}
  \MC^{\rho}(1) = \EC^{\rho}(1)\;\;\;\mbox{ and }\;\;\; \MC(w_0) = \EC^{\rho}(w_0),%
\end{equation*}
the fact that $\varphi^{\mathbb{P},\rho}_{t,u} = \varphi^{\mathbb{P},\rho,1}_{t,u_1} \tm \varphi_{t,u_2}^{\mathbb{P},\rho,2}$ implies%
\begin{equation*}
  \EC^-_1(u,x) = T_{x_1}\mathbb{P}\;\;\mbox{ and }\; \EC^0_1(u,x) = T_{x_2}\mathbb{P} \;\;\mbox{ for any }\left((x_1 x_2), (u_1,u_2)\right)\in\MC^{\rho}(1)%
\end{equation*}
and%
\begin{equation*}
  \EC^+_{w_0}(u,x) = T_{x_1}\mathbb{P}\;\;\mbox{ and }\; \EC^0_{w_0}(u,x) = T_{x_2}\mathbb{P} \;\;\mbox{ for any }\left((x_1,x_2), (u_1,u_2)\right)\in\MC^{\rho}(w_0).%
\end{equation*}
Consequently,%
\begin{align*}
 & \Lambda^-_{\Mo}(\MC^{\rho}(1)) \subset (-2-\ep,-2+\ep)\;\;\mbox{ and }\;\;\Lambda^0_{\Mo}(\MC^{\rho}(1))\subset (-\ep,\ep),\\
 & \Lambda^+_{\Mo}(\MC^{\rho}(w_0))\subset (-2-\ep,-2+\ep)\;\;\mbox{ and }\;\;\Lambda^0_{\Mo}(\MC^{\rho}(w_0))\subset (-\ep, \ep),%
\end{align*}
which by Theorem \ref{Teo1} implies that $E^{\rho}(1)$ and $E^{\rho}(w_0)$ are partially hyperbolic for small values of $\rho>0$.%

\begin{figure}[!ht]
	\begin{center}
		\includegraphics[scale=0.8,angle=270]{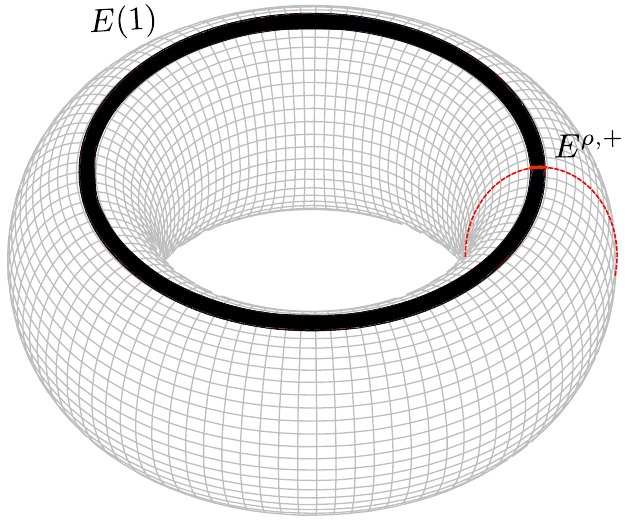}
	\end{center}
	\caption{Attractor chain control set}
	\label{fig2}
\end{figure}

\section{Invariance entropy}\label{sec_ie}

In this section, we present an application of our characterization of the Lyapunov spectra for the systems on $\F_{\Theta}$. First, we recall the concept of invariance entropy. Consider the control-affine system \eqref{eq_cas} and let $Q$ be an all-time controlled invariant set. For a compact set $K \subset Q$, the invariance entropy $h_{\inv}(K,Q)$ is defined as follows. For $\tau>0$, a set $\SC \subset \UC$ is called $(\tau,K,Q)$-spanning if for every $x\in K$ there is $u\in\SC$ with $\varphi([0,\tau],x,u) \subset Q$. Writing $r_{\inv}(\tau,K,Q)$ for the minimal cardinality of such a set, we put%
\begin{equation*}
  h_{\inv}(K,Q) := \limsup_{\tau \rightarrow +\infty}\frac{1}{\tau}\log r_{\inv}(\tau,K,Q) \in [0,\infty].%
\end{equation*}
There is an information-theoretic meaning of this quantity in the context of networked control, related to the practical stabilization of systems over digital communication channels. We refer the reader to \cite{Kaw} for more details.%

In \cite{DS3}, we have derived a lower bound on $h_{\inv}(K,Q)$ for a certain class of partially hyperbolic sets $Q$. To explain this result, we need to introduce some further notation: First, for every $u\in\UC$ we define the $u$-fiber of $Q$ by%
\begin{equation*}
  Q(u) := \left\{ x\in M\ :\ \varphi(\R,x,u) \subset Q \right\}.%
\end{equation*}
We also need to view the control flow $\phi_t:\UC \tm M \rightarrow \UC \tm M$ as a random dynamical system by discretizing the flow in time (with step size $t=1$) and equipping the base space $\UC$ with a $\theta_1$-invariant Borel probability measure $P$. Then, for every $\phi_1$-invariant probability measure $\mu$ on $\UC \tm M$ projecting to $P$, we can speak of the metric entropy $h_{\mu}(\varphi)$ of the RDS. If we start with an invariant measure $\mu$ on $\UC \tm M$, then the RDS is implicitly given by the projected measure $P := (\pi_{\UC})_*\mu$ and $h_{\mu}(\varphi)$ is defined by%
\begin{equation*}
  h_{\mu}(\varphi) := \sup_{\AC}\lim_{n\rightarrow+\infty}\frac{1}{n}\int_{\UC} H_{\mu_u}\Bigl(\bigvee_{i=0}^{n-1}\varphi_{i,u}^{-1}\AC\Bigr) \rmd P(u),%
\end{equation*}
where the supremum is taken over all finite measurable partitions of $M$ and $\{\mu_u\}_{u\in\UC}$ is the $P$-almost everywhere defined family of sample measures on $M$ so that $\rmd \mu(u,x) = \rmd \mu_u(x) \rmd P(u)$.%

Finally, for a partially hyperbolic all-time controlled invariant set $Q$, we introduce the unstable determinant%
\begin{equation*}
  J^+\varphi_{t,u}(x) := \left|\det (\rmd\varphi_{t,u})_{|E^+(u,x)}:E^+(u,x) \rightarrow E^+(\phi_t(u,x))\right|,\quad (u,x) \in \QC.%
\end{equation*}
The main result of \cite{DS3} then reads as follows.%

\begin{theorem}\label{thm_ielb}
Consider the control-affine system \eqref{eq_cas}. Let $Q \subset M$ be a compact all-time controlled invariant set with lift $\QC$, satisfying the following assumptions:%
\begin{enumerate}
\item[(A)] There exists a continuous invariant decomposition%
\begin{equation*}
  T_xM = E^{0-}(u,x) \oplus E^+(u,x),\quad \forall (u,x) \in \QC%
\end{equation*}
into subspaces $E^{0-}(u,x)$ and $E^+(u,x)$ of constant dimensions such that the following holds: There exists a constant $\lambda>0$ so that for every $\ep>0$ there is $T>0$ with%
\begin{align*}
  |(\rmd\varphi_{t,u})_xv| &\geq \rme^{\lambda t}|v| \mbox{\quad if\ } v \in E^+(u,x)\\
	|(\rmd\varphi_{t,u})_xv| &\leq \rme^{\ep t}|v| \mbox{\quad if\ } v \in E^{0-}(u,x)%
\end{align*}
for all $(u,x) \in \QC$ and $t \geq T$.%
\item[(B)] The set-valued map $u \mapsto Q(u)$ from $\UC$ into the compact subsets of $Q$ is lower semicontinuous.%
\item[(C)] The set $Q$ is isolated in the sense that there exists a neighborhood $N$ of $Q$ such that $\varphi(\R,x,u) \subset N$ implies $(u,x) \in \QC$ for any $(u,x) \in \UC \tm M$.%
\end{enumerate}
Then for all compact sets $K \subset Q$ of positive volume the invariance entropy satisfies%
\begin{equation*}
  h_{\inv}(K,Q) \geq \inf_{\mu \in M_{\phi_1}(\QC)}\left(\int_{\QC}\log J^+\varphi_{1,u}(x)\rmd\mu(u,x) - h_{\mu}(\varphi)\right),%
\end{equation*}
the infimum taken over all invariant measures $\mu$ of $\phi_1$ supported on $\QC$.%
\end{theorem}

\begin{lemma}\label{lem_fiber}
For any $u\in\UC$, the $u$-fiber of $E_{\Theta}(w)$ coincides with $\fix_{\Theta}(\mathsf{h}(u),w)$ and this set is a totally geodesic submanifold of $\F_{\Theta}$ with respect to an appropriately defined Riemannian metric (independent of $u$).%
\end{lemma}

\begin{proof}
The first statement follows from Theorem \ref{thm_morsesets}(i). For the second, see Lemma \ref{lem_totgeodesic}.%
\end{proof}


\begin{proposition}
Any chain control set $E_{\Theta}(w) \subset \F_{\Theta}$ satisfies the assumptions (B) and (C) in the above theorem.%
\end{proposition}

\begin{proof}
To verify assumption (B), let $u \in \UC$ and $x \in \fix_{\Theta}(\mathsf{h}(u),w)$. For any sequence $u_n \rightarrow u$ in $\UC$, pick $k_n,k \in K$ with $\mathsf{h}(u_n) = \Ad(k_n)H_{\phi}$, $\mathsf{h}(u) = \Ad(k)H_{\phi}$ and $x = k \cdot wb_{\Theta}$. The choice of such $k$ is possible, because $\mathsf{h}(u) = \Ad(k)H_{\phi}$ implies $\fix_{\Theta}(\mathsf{h}(u),w) = K_{\Ad(k)H_{\phi}} \cdot \Ad(k)wb_{\Theta}$, and hence $x = (k'k) \cdot wb_{\Theta}$, where $\Ad(k')$ fixes $\Ad(k)H_{\phi}$ implying $\mathsf{h}(u) = \Ad(k'k)H_{\phi}$.%

By continuity of $\mathsf{h}(\cdot)$, we have $\Ad(k_n)H_{\phi} = \mathsf{h}(u_n) \rightarrow \mathsf{h}(u) = \Ad(k)H_{\phi}$. Hence, every limit point of the sequence $(k^{-1}k_n)_{n\in\Z_+}$ is contained in $K_{H_{\phi}}$, and therefore $\dist(k^{-1}k_n,K_{H_{\phi}}) \rightarrow 0$. Since $k \mapsto k^{-1}$ is uniformly continuous on $K$, we also have $\dist(k_n^{-1}k,K_{H_{\phi}}) \rightarrow 0$. Hence, there are $z_n \in K_{H_{\phi}}$ with $d(k_n^{-1}k,z_n) \rightarrow 0$. Using the $K$-invariance of the metric on $\F_{\Theta}$, we find that $x_n := k_nz_n \cdot wb_{\Theta} \rightarrow k \cdot wb_{\Theta} = x$. Moreover, $x_n \in k_n(K_{H_{\phi}} \cdot wb_{\Theta}) = \fix_{\Theta}(\mathsf{h}(u_n),w)$, showing that $u \mapsto \fix_{\Theta}(\mathsf{h}(u),w)$ is lower semicontinuous.%

To show (C), we use that the sets $\MC_{\Theta}(w)$, $w\in\WC$, form a Morse decomposition for the control flow on $\UC \tm \F_{\Theta}$. We have $E_{\Theta}(w) = \pi_{\F_{\Theta}}(\MC_{\Theta}(w))$. Now consider some $(u,x) \in \UC \tm \F_{\Theta}$, not contained in a Morse set. Then the $\alpha$- and $\omega$-limit sets $\alpha(u,x)$ and $\omega(u,x)$ are contained in some $\MC_{\Theta}(w_1)$ and $\MC_{\Theta}(w_2)$, respectively, with $w_2 \notin \WC_{\Theta(\phi)} \backslash \WC / \WC_{\Theta}$. Hence, their projections to $\F_{\Theta}$ are contained in the corresponding (disjoint) chain control sets $E_{\Theta}(w_1)$ and $E_{\Theta}(w_2)$, respectively. Consequently, if $\varphi(\R,x,u)$ is contained in a neighborhood of $E_{\Theta}(w)$ whose closure intersects no other chain control set, then $\varphi(\R,x,u) \subset E_{\Theta}(w)$.%
\end{proof}

Now we can prove the main result of this section.%

\begin{theorem}
Assume that $\Lambda_{\Mo}^0(\MC_{\Theta}(w)) = \{0\}$. Then, for any compact set $K \subset E_{\Theta}(w)$ of positive volume%
\begin{equation}\label{eq_ie_lb}
  h_{\inv}(K,E_{\Theta}(w)) \geq \inf_{\mu \in M_{\phi_1}(\MC_{\Theta}(w))}\int \log J^+\varphi_{1,u}(x)\rmd \mu(u,x) = \inf_{(u,x) \in \MC_{\Theta}(w)} \limsup_{t \rightarrow +\infty}\frac{1}{t}\log J^+\varphi_{t,u}(x).%
\end{equation}
\end{theorem}

\begin{proof}
The theorem is proved in two steps.%

\emph{Step 1}: We apply Theorem \ref{thm_ielb}. To show that  Assumption (A) is satisfied, we put $\EC^{0-}_{\Theta,w}(u,x) := \EC^-_{\Theta,w}(u,x) \oplus \EC^0_{\Theta,w}(u,x)$, which obviously implies that%
\begin{equation*}
  T_x\F_{\Theta} = \EC^{0-}_{\Theta,w}(u,x) \oplus \EC^+_{\Theta,w}(u,x),\quad \forall (u,x) \in \MC_{\Theta}(w)%
\end{equation*}
is a continuous invariant decomposition. Observe that the dimensions of these subspaces are independent of $(u,x)$. According to Theorem \ref{central}(iii), the assumption $\Lambda_{\Mo}^0(\MC_{\Theta}(w)) = \{0\}$ implies%
\begin{equation}\label{eq_centralimpl}
  \lim_{n\rightarrow+\infty}\frac{1}{n}\max_{(u,x)\in\MC_{\Theta}(w)}\log\left\|(\rmd\varphi_{n,u})_{|\EC^0_{\Theta,w}(u,x)}\right\| = 0.%
\end{equation}
Now let $C,\lambda'>0$ be chosen so that%
\begin{equation*}
  |(\rmd\varphi_{t,u})_xv| \geq C\rme^{\lambda't}|v| \mbox{\quad for all\ } v \in \EC^+_{\Theta,w}(u,x)%
\end{equation*}
and fix some $\lambda \in (0,\lambda')$. Then, for any $\ep>0$ choose $T>0$ large enough so that%
\begin{equation*}
  \rme^{(\lambda'-\lambda)t} \geq C^{-1} \mbox{\quad and \quad} |(\rmd\varphi_{t,u})_xv| \leq \rme^{\ep t}|v|,\quad \forall v \in \EC^{0-}_{\Theta,w}(u,x) \mbox{\ and\ } t \geq T.%
\end{equation*}
The existence of such $T$ immediately follows from \eqref{eq_centralimpl} and $\lambda < \lambda'$. Hence, (A) holds and we have%
\begin{equation*}
  h_{\inv}(K,E_{\Theta}(w)) \geq \inf_{\mu \in M_{\phi_1}(\QC)}\left(\int_{\QC}\log J^+\varphi_{1,u}(x)\rmd\mu(u,x) - h_{\mu}(\varphi)\right).%
\end{equation*}

\emph{Step 2}: We prove that $h_{\mu}(\varphi) = 0$ for all $\mu$. To this end, observe that $h_{\mu}(\varphi)$ is bounded from above by the topological entropy $h_{\tp}(\varphi)$ of the corresponding bundle RDS on $\MC_{\Theta}(w)$, following from the variational principle for bundle RDS (cf.~\cite[Thm.~1.2.13]{HKa}). We show that $h_{\tp}(\varphi) = 0$ for each fixed invariant measure $P$ on the base space $\UC$. Consider two points $x,y$ on the same fiber $\fix_{\Theta}(\mathsf{h}(u),w)$, $u\in\UC$. Since $\fix_{\Theta}(\mathsf{h}(u),w)$ is totally geodesic by Lemma \ref{lem_fiber}, we can take a shortest geodesic $\gamma:[0,1] \rightarrow \fix_{\Theta}(\mathsf{h}(u),w)$ from $x$ to $y$. Then for each $t > 0$ we have%
\begin{align*}
  d(\varphi_{t,u}(x),\varphi_{t,u}(y)) &\leq \length(\varphi_{t,u} \circ \gamma) = \int_0^1 |(\rmd\varphi_{t,u})_{\gamma(s)}\dot{\gamma}(s)| \rmd s \\
	 &\leq \sup_{z \in \fix_{\Theta}(\mathsf{h}(u),w)\atop v\in T_z\fix_{\Theta}(\mathsf{h}(u),w),\ |v| = 1} |(\rmd\varphi_{t,u})_zv| \cdot \length(\gamma).%
\end{align*}
Now observe that $\length(\gamma) = d(x,y)$ and $T_z\fix_{\Theta}(\mathsf{h}(u),w) = \EC^0_{\Theta,w}(u,z)$ by \cite[Prop.~4.2]{DS2}. By (A) we can thus choose $t$ large enough (independently of $u,z$ and $v$) so that $|(\rmd\varphi_{t,u})_zv| \leq \rme^{\ep t}|v|$. Hence,%
\begin{equation*}
  d(\varphi_{t,u}(x),\varphi_{t,u}(y)) \leq \rme^{\ep t} d(x,y).%
\end{equation*}
By standard methods, one shows that this implies $h_{\tp}(\varphi) \leq \ep \cdot \dim \fix_{\Theta}(\mathsf{h}(u),w)$, and since $\ep>0$ was chosen arbitrarily, $h_{\tp}(\varphi) = 0$ follows. The equality in \eqref{eq_ie_lb} follows from the general theory of continuous additive cocycles, see, e.g., \cite{SMLS}.%
\end{proof}

\appendix
	
\section{Appendix}\label{sec_ap}
	
\subsection{Regular sequences}

In this section, we show how we can obtain the asymptotic ray of a given sequence by means of its Cartan decomposition. Although such a result was used in \cite{ASM}, we could not find its proof in the literature.%

Let $G$ be a noncompact semisimple Lie group with finite center and consider the left coset symmetric space $K\backslash G$. Let $d$ be the $G$-invariant distance in $K \backslash G$, which is uniquely determined by%
\begin{equation*}
  d(o \cdot \exp X,o) = |X|\quad \forall X \in \fs,%
\end{equation*}
where $o = K \cdot e$ is the origin of $K \backslash G$ and $|\cdot|$ is the $K$-invariant norm induced by $B_{\zeta}$.%

Following \cite{Kai}, a sequence $(g_n)$ in $G$ is called \emph{regular} if there exists $D \in \fs$ such that $d(o \cdot g_n,o \cdot \exp nD)$ has sublinear growth as $n \rightarrow +\infty$, i.e.,%
\begin{equation*}
  \lim_{n \rightarrow +\infty}\frac{1}{n}d(o \cdot g_n,o \cdot \exp nD) = 0.%
\end{equation*}
If such $D$ exists, it is unique and called the \emph{asymptotic ray} of $(g_n)$.%

The next result yields an expression for the asymptotic ray of a regular sequence in terms of its polar decomposition.%
	
\begin{lemma}\label{asymptotic}
If $(g_n)$ is a regular sequence, then the asymptotic ray of $(g_n)$ is given by 
\begin{equation*}
  D = \lim_{n\rightarrow+\infty}\frac{1}{n}\log \rmS(g_n).%
\end{equation*}
\end{lemma}
		
\begin{proof}
By the uniqueness of the asymptotic ray, we only have to show that $\frac{1}{n}d(o\cdot g_n,o\cdot \rme^{n D})\rightarrow 0$. If $\psi_n:=\Ad(g_n)$, \cite[Thm.~2.3 and Thm.~4.1]{Kai} imply that $(\psi_n)$ is also a regular sequence such that%
\begin{equation*}
  \Psi := \lim_{n\rightarrow+\infty}(\psi_n^*\psi_n)^{1/2n}\;\;\mbox{ exists and }\;\; \lim_{n\rightarrow+\infty}\frac{1}{n}\log\|\psi_n\Psi^{-n}\| = 0.%
\end{equation*}
On the other hand, by the $K$-invariance of the inner product, $\Ad(k)^* = \Ad(k)^{-1}$. Hence, if $g_n = k_n\rmS(g_n)$ is the Cartan decomposition of $g_n$, we obtain $\psi_n^*\psi_n = \Ad(\rmS(g_n)^2)$, implying%
\begin{equation*}
  (\psi_n^*\psi_n)^{1/2n} = \Ad(\rmS(g_n))^{1/n} = \Ad\left(\exp\left(\frac{1}{n}\log\rmS(g_n)\right)\right),%
\end{equation*}
where for the last equality we use that $\Ad(\rmS(g_n))$ admits a unique $n$th root, since it is a positive definite self-adjoint
linear map (see \cite[Thm.~7.2.6]{HJo}). Moreover, the fact that $\Ad \circ \exp$ is a homeomorphism when restricted to $\fs$ yields that%
\begin{equation*}
  D = \lim_{n \rightarrow +\infty}\frac{1}{n}\rmS(g_n) \mbox{\quad exists and satisfies\ } \Psi = \Ad(\rme^D).%
\end{equation*}
Hence,
\begin{align*}
  d(o \cdot g_n, o \cdot \rme^{nD}) &= d(o \cdot g_n\rme^{-nD},o) = \left|\log A^+\left(g_n\rme^{-nD}\right)\right| \leq \left\|\ad\left(\log A^+\left(g_n\rme^{-nD}\right)\right)\right\|\\
			&= \log\left\|\Ad\left(A^+\left(g_n\rme^{-nD}\right)\right)\right\| = \log\left\|\Ad\left(g_n\rme^{-nD}\right)\right\| = \log\left\|\psi_n\Psi^{-n}\right\|,%
\end{align*}
where for the inequality we used that the positive roots generate $\fa^*$. Consequently,%
\begin{equation*}
  \lim_{n\rightarrow+\infty}\frac{1}{n}d(o \cdot g_n, o\rme^{nD}) \leq \lim_{n\rightarrow+\infty}\frac{1}{n}\log\|\psi_n\Psi^{-n}\| = 0,%
\end{equation*}
concluding the proof.
\end{proof}

\subsection{Totally geodesics submanifolds}

Let $H\in\cl(\fa^+)$ and consider its action on $\F_{\Theta}$. Here we show that for a suitable $K$-invariant Riemannian metric, the sets $\fix_{\Theta}(H,w)$, $w\in\WC,$ are totally geodesic submanifolds in the sense that any two points in $\fix_{\Theta}(H,w)$ can be joined by a geodesic of $G/P_{\Theta}$ whose image lies in $\fix_{\Theta}(H,w)$.%

For a given $\Theta\subset\Sigma$, let us consider the homogeneous space $K/K_{\Theta}$ and the map $\pi:G/P_{\Theta}\rightarrow K/K_{\Theta}$ given by $\pi(g\cdot b_{\Theta}) := \kappa(g)\cdot o$, where $o=e\cdot K_{\Theta}$ and $\kappa:G\rightarrow K$ is the map that assigns to any $g\in G$ its $K$-component in the Iwasawa decomposition. It is not hard to see that $\pi$ is well-defined, commutes with the action of $K$ and has an inverse given by $k\cdot o\in K/K_{\Theta}\mapsto k\cdot b_{\Theta}\in G/P_{\Theta}$. Moreover, since both maps are quotient maps and $\kappa$ is differentiable (see \cite[Thm.~6.46]{Kna}), $\pi$ is a diffeomorphism.%

\begin{lemma}\label{lem_totgeodesic}
There is a $K$-invariant metric in $G/P_{\Theta}$ such that the sets $\fix_{\Theta}(H,w)$, $w\in\WC$, are totally geodesic submanifolds. 
\end{lemma}

\begin{proof}
By the previous discussion, the pullback of any $K$-invariant Riemannian metric on $K/K_{\Theta}$ by $\pi$ is a $K$-invariant metric on $G/P_{\Theta}$. Since for such a metric $\pi$ is an isometry and%
\begin{equation*}
  \pi(\fix_{\Theta}(H,w)) = \pi(K_H\cdot wb_{\Theta}) = \pi(w(K_{w^{-1}H})\cdot b_{\Theta}) = wK_{w^{-1}H}\cdot o,%
\end{equation*}
it is enough to show that there is a $K$-invariant metric on $K/K_{\Theta}$ such that $K_{wH}\cdot o$, $w\in\WC$, is totally geodesic, which we will do in 3 steps.%

\emph{Step 1}: We prove that for any $\Theta\subset \Sigma$ it holds that%
\begin{equation*}
  \fk_{\Theta}=\fm\oplus\sum_{\alpha\in\langle\Theta\rangle}(\fg_{\alpha}+\fg_{-\alpha})\cap\fk\;\;\;\;\mbox{ and }\;\;\;\;\fk_{\Theta}^{\perp}=\sum_{\alpha\in\Pi\setminus\langle\Theta\rangle}(\fg_{\alpha}+\fg_{-\alpha})\cap\fk,%
\end{equation*}
where the inner product considered here is the restriction of $B_{\zeta}$ to $\fk$.%

If we consider $H_{\Theta}\in\cl(\fa^+)$ such that $\Theta = \Theta(H_{\Theta})$, we have $\fk_{\Theta} = \{X\in\fk: \;[H_{\Theta}, X]=0\}$. Therefore, the inclusion $\fk_{\Theta}\supset\fm\oplus\sum_{\alpha\in\langle\Theta\rangle}(\fg_{\alpha}+\fg_{-\alpha})\cap\fk$ certainly holds, since the bracket of any element in the right-hand side with $H_{\Theta}$ is zero. On the other hand, any $X\in\fk_{\Theta}$ can be written as $X=\sum_{\alpha\in\Pi\cup\{0\}}X_{\alpha}$, and therefore%
\begin{equation*}
  0 = [H_{\Theta},X] = \sum_{\alpha\in\Pi\setminus\langle\Theta\rangle}\alpha(H)X_{\alpha}.%
\end{equation*}
Since $\alpha(H)\neq0$, we get $X_{\alpha}=0$ for any $\alpha\in\Pi\setminus\langle\Theta(H)\rangle$, implying the first equality.%
	
The second equality follows from the fact that the vector subspaces of $\fg$ given by%
\begin{equation*}
  \fm\oplus\sum_{\alpha\in\langle\Theta(H)\rangle}(\fg_{\alpha}+\fg_{-\alpha})\;\;\;\;\mbox{ and }\;\;\;\;\sum_{\alpha\in\Pi\setminus\langle\Theta(H)\rangle}(\fg_{\alpha}+\fg_{-\alpha})%
\end{equation*}
are orthogonal w.r.t.~the inner product $B_{\zeta}$ and their direct sum contains $\fk$.%
	
\emph{Step 2}: We prove that for any $\Theta_1, \Theta_2\subset\Sigma$ and $w\in\WC$ it holds that%
\begin{equation*}
  w\fk_{\Theta_2}=(w\fk_{\Theta_2}\cap \fk_{\Theta_1})\oplus(w\fk_{\Theta_2}\cap\fk_{\Theta_1}^{\perp}).%
\end{equation*}
In particular, the orthogonal complement in $w\fk_{\Theta_2}$ of $w\fk_{\Theta_2}\cap\fk_{\Theta_1}$ w.r.t.~$B_{\zeta}|_{w\fk_{\Theta_2}\times w\fk_{\Theta_2}}$ is $w\fk_{\Theta_2}\cap\fk_{\Theta_1}^{\perp}$.%
	
Since $w\Pi=\Pi$, $w\fm=\fm$, $w\fk=\fk$ and $w\fg_{\alpha}=\fg_{w\alpha}$, by Step 1 we obtain%
\begin{align*}
&  w\fk_{\Theta_2}=w\fm \oplus\sum_{\alpha\in \langle\Theta_2\rangle}(w\fg_{\alpha}+w\fg_{-\alpha})\cap w\fk=
	\fm \oplus\sum_{\alpha\in w\langle\Theta_2\rangle}(\fg_{\alpha}+\fg_{-\alpha})\cap\fk\\
	&=\left(\fm \oplus\sum_{\alpha\in w\langle\Theta_2\rangle\cap \langle\Theta_1\rangle}(\fg_{\alpha}+\fg_{-\alpha})\cap\fk\right)+\left(\sum_{\alpha\in w\langle\Theta_2\rangle\cap\Pi\setminus\langle\Theta_1\rangle}(\fg_{\alpha}+\fg_{-\alpha})\cap\fk\right)\subset (w\fk_{\Theta_2}\cap \fk_{\Theta_1})\oplus (w\fk_{\Theta_2}\cap\fk_{\Theta_1}^{\perp}).%
\end{align*}
Since the reverse inequality always holds, we are done.%
	
\emph{Step 3}: We prove that there is a $K$-invariant Riemannian metric on $K/K_{\Theta}$ such that $K_{wH}\cdot o$ is totally geodesic for any $w\in\WC$.%
	
The inner product $(B_{\zeta})_{|\fk\tm\fk}$ is $K$-invariant, since the Cartan-Killing form is invariant by automorphisms, and hence, it induces on $K/K_{\Theta}$ a $K$-invariant metric in the following way: Since $T_o (K/K_1)\simeq\fk_{\Theta}^{\perp}$, we define%
\begin{equation*}
  \langle v, w\rangle_x:=\langle(d[k])^{-1}_xv, (d[k])^{-1}_xw\rangle, \;\;\mbox{ where }\;\;x=k\cdot o.%
\end{equation*}
It is a well-known fact that for such a Riemannian metric, the geodesics starting at the origin are given by $\gamma(t)=\rme^{tX}\cdot o$ with $X\in \fk_{\Theta}^{\perp}$ (see \cite[Prop.~25]{Onell}). Moreover, since $K/K_{\Theta}$ is compact, it is geodesically complete and therefore, the geodesic connecting the origin to any given point $x\in K/K_{\Theta}$ is by uniqueness of the form $\gamma(t)=\rme^{tX}\cdot o$ for some $X\in\fk_{\Theta}^{\perp}$. Since the metric is $K$-invariant, we obtain that any given points $x_1=k_1\cdot o$ and $x_2=k_2\cdot o$ in $K/K_{\Theta}$ can be joined by a geodesic $\alpha(t)=k_1\gamma(t)$, where $\gamma(t)=\rme^{tX}\cdot o$ is a geodesic joining the origin and $k_1^{-1}k_2\cdot o$.%
	
On the other hand, by Step 2 it holds that $T_o(K_{wH}\cdot o)=w\fk_H\cap\fk_{\Theta}^{\perp}$ and consequently $K_{wH}\cdot o$ is isometric to the homogeneous space $K_{wH}\big/(K_{wH}\cap K_{\Theta})$, where the Riemannian metric of the homogeneous space is the $K_{wH}$-invariant Riemannian metric induced by $(B_{\zeta})_{|w\fk_H\times w\fk_H}$. As above, any two points in $K_{wH}\cdot o$ can be joined by a geodesic of the form $\gamma(t)=k\rme^{tX}\cdot o$ for some $k\in K_{wH}$ and $X\in w\fk_H\cap\fk_{\Theta}^{\perp}$, and since $w\fk_{H}\cap \fk_{\Theta}^{\perp}\subset\fk_{\Theta}^{\perp}$, such geodesics are also geodesics of $K/K_{\Theta}$, showing that $K_{wH}\cdot o$ is totally geodesic as stated.
\end{proof}

\subsection{The multiplicative ergodic theorem}\label{subsec_met}

This section is devoted to the presentation of the multiplicative ergodic theorem (MET), also known as Oseledets theorem. For more details, the reader should consult \cite[Ch.~3]{LAr} or \cite[Ch.~11]{CK2}.%

A metric dynamical system $(\Omega,\mathcal{F},\nu,(\theta_n)_{n\in\T})$ is given by a probability space $(\Omega,\mathcal{F},\nu)$ and a measurable (semi-) flow $(\theta_n)_{n\in\T}$ for which $\mu$ is an invariant measure, where ($\T=\Z_+$) $\T=\Z$.%

Let $E$ be a $d$-dimensional Euclidean vector space. For any given metric dynamical system $(\Omega,\mathcal{F},\nu,(\theta_n)_{n\in\T})$ and any random map $A:\Omega\rightarrow\Gl(E)$, we can define a linear cocycle $\psi$ on $E$ by $\psi(0,\omega) := \id_E$, $\psi(n,\omega) := A(\theta_{n-1}\omega)\cdots A(\omega)$ if $n>0$ and, in case $\T=\Z$, $\psi(n,\omega) := A^{-1}(\theta_{n}\omega)\cdots A^{-1}(\theta_{-1}\omega)$ for all $n<0$.%

The proof of the next result can be found in \cite[Thm.~3.4.2]{LAr}.%

\begin{theorem}\label{MET}
Let $\psi$ be a linear cocycle on the vector space $E$ over the metric dynamical system $(\Omega,\mathcal{F},\mu,(\theta_n)_{n\in\T})$ with generator $A:\Omega\rightarrow \Gl(E)$. If%
\begin{equation*}
  \log^+\|A\| \in L^1(\Omega,\mathcal{F},\mu) \;\;\mbox{ and }\;\;\log^+\|A^{-1}\| \in L^1(\Omega,\mathcal{F},\mu),%
\end{equation*}
then there exists an invariant set $\tilde{\Omega}\in\mathcal{F}$ such that for each $\omega\in\tilde{\Omega}$ the following statements hold:%
\begin{itemize}
		\item[(A)] The case $\T=\Z_+$: 
		\subitem(i) The limit $\lim_{n\rightarrow+\infty}(\psi(n, \omega)^* \psi(n, \omega))^{1/2n} =: \Psi(\omega)\geq 0$ exists;
		\subitem(ii) Let $\rme^{\lambda_{p(\omega)}(\omega)}<\ldots<\rme^{\lambda_1(\omega)}$ be the different eigenvalues of $\Psi(\omega)$ and let $U_{p(\omega)}(\omega), \ldots, U_1(\omega)$ be their corresponding eigenspaces with multiplicities $d_i(\omega)=\dim U_i(\omega)$. Then%
$$p(\theta\omega)=p(\omega),$$
		$$\lambda_i(\theta\omega)=\lambda_i(\omega) \;\;\mbox{ for all }\;\;i\in\{1, \ldots, p(\omega)\},$$
		$$d_i(\theta\omega)=d_i(\omega)\;\;\mbox{ for all }\;\;i\in\{1, \ldots, p(\omega)\}.$$
		\subitem(iii) Put $V_{p(\omega)+1}(\omega)=\{0\}$ and for $i=1,\ldots,p(\omega)$, 
		$$V_i(\omega) := U_{p(\omega)}(\omega)\oplus\cdots\oplus U_i(\omega),\;\;\; $$
		so that%
		$$V_{p(\omega)}(\omega) \subset\cdots \subset V_i(\omega)\subset\cdots\subset V_1(\omega)=E$$
		defines a filtration of $E$. Then for each $v\in E\setminus\{0\}$ the Lyapunov exponent
		$$\lambda(\omega,v) := \lim_{n\rightarrow+\infty}\frac{1}{n}\log|\psi(n,\omega)v|$$
		exists as a limit and%
		\begin{equation}\label{eq_lyapexp_subspace}
		  \lambda(\omega,v) = \lambda_i(\omega)\iff v\in V_i(\omega)\setminus V_{i+1}(\omega),%
		\end{equation}
		or equivalently
		$$V_i(\omega)=\{v\in V: \;\lambda(\omega, v)\leq\lambda_i(\omega)\}.$$
		\subitem(iv) For all $v\in E\setminus\{0\}$
		$$\lambda(\theta\omega, A(\omega)v)=\lambda(\omega, v),$$
		whence%
		$$A(\omega)V_i(\omega) = V_i(\theta\omega) \;\;\mbox{ for all }\;\;i\in\{1,\ldots,p(\omega)\}.$$
		\subitem(v) The function $\omega\mapsto p(\omega)\in\{1,\ldots,d\}$ (measurably extended from $\tilde{\Omega}$ to $\Omega$) is measurable. The functions $\omega\mapsto \lambda_i(\omega)\in\R$, $\omega\mapsto d_i(\omega)\in\{1, \ldots, d\}$, $\omega\mapsto U_i(\omega)\in \cup_{k=1}^d\mathrm{G}_k(d)$ and $\omega\mapsto V_i(\omega)\in \cup_{k=1}^d\mathrm{G}_k(d)$, $\mathrm{G}_k(d)$ the Grassmannian manifold of the $k$-dimensional subspaces of $E$ (measurably extended to $\{\omega: p(\omega)\geq i\}\in\mathcal{B}$) are measurable. Moreover, if $(\Omega,\mathcal{F},\mu,(\theta_n)_{n\in\T})$ is ergodic, then the functions $p(\cdot)$, $\lambda_i(\cdot)$ and $d_i(\cdot)$ are constant.
		\item[(B)] The case $\T=\Z$: There exists a splitting 
		$$E=E_1(\omega)\oplus \cdots \oplus E_{p(\omega)}(\omega)$$
		of $E$ into random subspaces $E_i(\omega)$ (called Oseledets spaces) depending measurably on $\omega$ with dimensions $\dim E_i(\omega)=d_i(\omega)$, satisfying:
		\subitem(i) If $P_i(\omega): E\rightarrow E_i(\omega)$ denotes the projection onto $E_i(\omega)$ along $F_i(\omega):=\sum_{j\neq i}E_j(\omega)$, then
		$$A(\omega)P_i(\omega)=P_i(\theta\omega)A(\omega)$$
		or equivalently 
		$$A(\omega)E_i(\omega)=E_i(\theta\omega).$$
		\subitem(ii) We have
		$$\lim_{n\rightarrow\pm\infty}\frac{1}{n}\log\|\psi(n, \omega)v\|=\lambda_i(v)\iff v\in E_i(\omega)\setminus\{0\}.$$
		\subitem(iii) Convergence in (ii) is uniform with respect to $v\in E_i(\omega)\cap S^{d-1}$ for each fixed $\omega$. 
		\subitem(iv) The filtration in (A) can be recovered as%
		$$V_i(\omega)=\oplus_{j=i}^{p(\omega)}E_j(\omega) \;\mbox{ for }i\in\{1, \ldots, p(\omega)\}.$$		
\end{itemize}
\end{theorem}

\begin{remark}
Although for $\omega\in\tilde{\Omega}$ the invariant splitting%
\begin{equation*}
  E = E_1(\omega)\oplus\cdots\oplus E_{p(\omega)}(\omega)%
\end{equation*}
is not in general orthogonal (and the splitting $E = U_1(\omega)\oplus\cdots\oplus U_{p(\omega)}(\omega)$ not in general invariant), for any fixed $\kappa>0$ there is a random inner product $\langle\cdot,\cdot\rangle_{\kappa,\omega}$ satisfying (see \cite[Thm.~4.3.6]{LAr}):%
\begin{itemize}
\item[1.] $\langle\cdot,\cdot\rangle_{\kappa,\omega}$ depends measurably on $\omega$ and each $P_i(\omega)$ is an orthogonal projection;%
\item[2.] For any $\varepsilon>0$ there exists a random variable $B_{\varepsilon}:\Omega\rightarrow [1,+\infty)$ such that%
\begin{equation*}
  B_{\varepsilon}(\omega)^{-1}\|\cdot\|\leq \|\cdot\|_{\kappa,\omega}\leq B_{\varepsilon}(\omega)\|\cdot\|\;\;\;\mbox{ with }\;\;\;\rme^{-\varepsilon|n|} B_{\varepsilon}(\omega)\leq B_{\varepsilon}(\theta_n\omega)\leq \rme^{\varepsilon|n|}B_{\varepsilon}(\omega) \;\mbox{ for all }\;n\in\Z.%
\end{equation*}
\item[3.] For all $i \in \{1,\ldots,p(\omega)\}$, $v\in E_i(\omega)$ and $n\in\Z$ it holds that%
\begin{equation*}
  \rme^{n\lambda_i(\omega) - \kappa|n|}\|v\|_{\kappa,\omega} \leq \|\psi(n,\omega)v\|_{\kappa,\theta_n\omega} \leq \rme^{n\lambda_i(\omega) + \kappa|n|}\|v\|_{\kappa,\omega}.%
\end{equation*}
\end{itemize}
\end{remark}

We have the following result:%

\begin{lemma}\label{liminf}
Let $\omega\in\tilde{\Omega}$ and consider an isomorphism $T:E\rightarrow E$ such that $TV_j(\omega)=V_j(\omega)$ for all $j \in \{1,\ldots,p(\omega)\}$. If $V$ is a $\Psi(\omega)$-invariant proper subspace of $E$, then the following statements hold:%
\begin{itemize}
\item[(i)] For any nonzero vector $u\in V^{\perp}\cap (V_i(\omega)\setminus V_{i+1}(\omega))$ we have%
\begin{equation*}
  \lim_{n\rightarrow+\infty}\frac{1}{n}\log\inf_{v\in V}\left|\psi(n,\omega)T(u-v)\right| = \lambda_i(\omega).%
\end{equation*}
\item[(ii)] Let $W \subset E$ be a $\Psi(\omega)$-invariant subspace such that $V^{\perp} \cap W \neq \{0\}$. Let $i^* \in \{1,\ldots,p(\omega)\}$ be the greatest index such that $U_{i^*}(\omega)\cap (V^{\perp} \cap W) \neq \{0\}$. Then, for any $\varepsilon>0$ and any positive real numbers $c_2>c_1>0$ there exists $C = C(\varepsilon,c_1,c_2)>0$ such that%
\begin{equation*}
  \inf_{u\in R, v\in V}\left|\psi(n,\omega) T(u-v)\right|>C\rme^{n(\lambda_{i^*}(\omega)-\varepsilon)} \;\;\mbox{ for all }n > 0,%
\end{equation*}
where $R = \{ u\in V^{\perp} \cap W: c_1\leq |u|\leq c_2\}$.%
\end{itemize}
\end{lemma}

\begin{proof}
(i) The facts that $\inf_{v\in V}\left|\psi(n,\omega) T(u-v)\right|\leq \left|\psi(n,\omega) Tu\right|$ and $TV_j(\omega)=V_j(\omega)$ for all $j\in \{1, \ldots, p(\omega)\}$ together with \eqref{eq_lyapexp_subspace} imply%
\begin{equation}\label{eq_limsup_ue}
  \limsup_{n\rightarrow+\infty}\frac{1}{n}\log\inf_{v\in V}\left|\psi(n,\omega)T(u-v)\right| \leq \lambda_i.%
\end{equation}
On the other hand, let $\kappa>0$ and consider the random inner product $\langle\cdot,\cdot\rangle_{\kappa,\omega}$ as above. By the invariance of the decomposition, we have%
\begin{align*}
  |\psi(n,\omega)T(u-v)|_{\kappa, \theta_n\omega}^2 &= \left|\sum_{j=1}^{p(\omega)}\psi(n,\omega)P_j(\omega)T(u-v)\right|_{\kappa,\theta_n\omega}^2 = \sum_{j=1}^{p(\omega)}\Bigl|\psi(n, \omega)P_j(\omega)T(u-v)\Bigr|_{\kappa,\theta_n\omega}^2\\
	&\geq \sum_{j=1}^{i}\Bigl|\psi(n, \omega)P_j(\omega)T(u-v)\Bigr|_{\kappa,\theta_n\omega}^2 \geq \sum_{j=1}^{i}\rme^{2(n\lambda_j-\kappa|n|)}\bigl|P_j(\omega)T(u-v)\bigr|_{\kappa,\omega}^2\\
	&\geq \rme^{2(n\lambda_i-\kappa|n|)}\bigl|Q^i(\omega)T(u-v)\bigr|_{\kappa,\omega}^2,%
\end{align*}
where $Q^i(\omega)=\sum_{j=1}^iP_j(\omega)$. Let $\varepsilon>0$ and consider the random variable $B_{\varepsilon}:\Omega\rightarrow [1, +\infty)$ as in the above remark. Then%
\begin{align*}
  |\psi(n,\omega)T(u-v)| &\geq |\psi(n,\omega)T(u-v)|_{\kappa,\theta_n\omega} B_{\varepsilon}(\theta_n\omega)^{-1}\geq \rme^{n\lambda_i-\kappa|n|}\bigl|Q^i(\omega)T(u-v)\bigr|_{\kappa,\omega} B_{\varepsilon}(\theta_n\omega)^{-1}\\
	&\geq \rme^{n\lambda_i-\kappa|n|}\bigl|Q^i(\omega)T(u-v)\bigr| B_{\varepsilon}(\theta_n\omega)^{-1} B_{\varepsilon}(\omega)^{-1}\\
	&\geq \rme^{n\lambda_i-\kappa|n|}\bigl|Q^i(\omega)T(u-v)\bigr|\rme^{-\varepsilon|n|}B_{\varepsilon}(\omega)^{-2},%
\end{align*}
and so%
\begin{equation*}
  \inf_{v\in V}|\psi(n, \omega)T(u-v)|\geq\rme^{n\lambda_i-(\kappa+\varepsilon)|n|}\inf_{v\in V}\bigl|Q^i(\omega)T(u-v)\bigr|B_{\varepsilon}(\omega)^{-2}.%
\end{equation*}

\emph{Claim}: $\inf_{v\in V}\bigl|Q^i(\omega)T(u-v)\bigr|\neq 0$.

By the equivalence of the norms, it is enough to show that $\inf_{v\in V}\bigl|Q^i(\omega)T(u-v)\bigr|_{\kappa,\omega}\neq 0$. It is easy to see that the infimum is indeed a minimum. However, if for some $v\in V$ we have $\bigl|Q^i(\omega)T(u-v)\bigr|_{\kappa,\omega}=0$, then%
\begin{equation*}
  T(u-v) \in \ker Q^i(\omega) = \bigoplus_{j=i+1}^{p(\omega)}E_j(\omega) = V_{i+1}(\omega)\implies u-v\in T^{-1}V_{i+1}(\omega)=V_{i+1}(\omega).%
\end{equation*}
Given the fact that $\Psi(\omega)$ is a self-adjoint map and $\Psi(\omega)V = V$, we have%
\begin{equation}\label{selfadjoint}
  V = \bigoplus_{j=1}^{p(\omega)} \left(V\cap U_j(\omega)\right) \;\;\;\;\mbox{ and }\;\;\;\; V^{\perp}=\bigoplus_{j=1}^{p(\omega)} \left(V^{\perp}\cap U_j(\omega)\right).%
\end{equation}
Moreover, since $u\in V^{\perp}$ and $v\in V$, decomposition (\ref{selfadjoint}) implies that $u\in V_{i+1}(\omega)$ if $u-v\in V_{i+1}(\omega)$, contradicting our hypothesis. Therefore, for any $v\in V$ it holds that $|Q(\omega)T(u-v)|_{\kappa,\omega}>0$, proving the claim.%

From the above claim we obtain%
\begin{equation*}
  \liminf_{n\rightarrow+\infty}\frac{1}{n}\inf_{w\in W}|\psi(n,\omega)T(v-w)| \geq \lambda_i-(\kappa+\varepsilon),%
\end{equation*}
and since $\kappa,\varepsilon>0$ are arbitrary, together with \eqref{eq_limsup_ue}, statement (i) follows.%
	
(ii) By decomposition \eqref{selfadjoint} and the $\Psi(\omega)$-invariance of $W$, any $u\in V^{\perp} \cap W$ satisfies $u\in V_i(\omega)\setminus V_{i+1}(\omega)$ for some $i\leq i^*$. Therefore, for any given $u\in V^{\perp} \cap W$ we can show exactly as in the proof of item (i) that%
\begin{equation*}
  \inf_{v\in V}|\psi(n,\omega)T(u-v)| \geq \rme^{n\lambda_{i^*}-\varepsilon|n|}\inf_{v\in V}\bigl|Q^{i^*}(\omega)T(u-v)\bigr|B_{\varepsilon}(\omega)^{-2}.%
\end{equation*}
The result is proven if we can show that%
\begin{equation*}
  C = C(\varepsilon,c_1,c_2) := \inf_{u\in V^{\perp} \cap W; \,c_1\leq |u|\leq c_2}\left\{\inf_{v\in V}\bigl|Q^{i^*}(\omega)T(u-v)\bigr|B_{\varepsilon}(\omega)^{-2}\right\} > 0.%
\end{equation*}
To see this, observe that if $P$ denotes the orthogonal projection onto $Q^{i^*}(\omega)TV$, then%
\begin{equation*}
  \inf_{v\in V}\bigl|Q^{i^*}(\omega)T(u-v)\bigr| = |(I-P)Q^{i^*}(\omega)Tu|%
\end{equation*}
and, as proved in item (i), this number is positive for any $u$. Then compactness of $R$ implies the existence of $u_*\in R$ such that
\begin{equation*}
  C = \inf_{v\in V}\bigl|Q^{i^*}(\omega)T(u_*-v)\bigr|B_{\varepsilon}(\omega)^{-2} > 0,%
\end{equation*}
concluding the proof.	
\end{proof}

\section*{Acknowledgements}

The authors express their gratitude to Lino Grama for his help with the proof of Lemma \ref{lem_totgeodesic}.%

\end{document}